 \documentclass[final,3p,times]{elsarticle}

 \usepackage{etex}

\usepackage{amssymb}
\usepackage{amsmath}
\usepackage{amsfonts}
\usepackage{mathtools}
\usepackage{bm}
\usepackage{tikz}
\usepackage{dsfont}
\usetikzlibrary{trees,shapes}
\usepackage{pifont}
\usepackage{pgfplots}
\usepackage{subcaption}
\usepackage[all]{xy}

\usepackage{amsthm}

\newdefinition{example}{Example}
\newtheorem{theorem}{Theorem}

\newdefinition{definition}{Definition}
\newtheorem{corollary}{Corollary}

\newcommand\independent{\protect\mathpalette{\protect\independenT}{\perp}}
\def\independenT#1#2{\mathrel{\rlap{$#1#2$}\mkern2mu{#1#2}}}
\usepackage{verbatim}
\usepackage{array,amsmath,booktabs}
\usepackage{sansmath}
\usepackage{hyperref}
\usepackage{natbib}
\pgfplotsset{compat=newest}
 
 \usepackage{bbm}

\begin{document}

\begin{frontmatter}


\title{A geometric characterization of sensitivity analysis in monomial models}


\author{Manuele Leonelli\footnote{Corresponding Author. Full address: School of Human Sciences and Technology, Calle Maria de Molina 6, 28006 Madrid, Spain. Email: manuele.leonelli@ie.edu.}}

\address{School of Human Sciences and Technology, IE University, Spain}
\author{Eva Riccomagno}

\address{Dipartimento di Matematica, Universit\'{a} degli Studi di Genova, Italia.}

\begin{abstract}
Sensitivity analysis in probabilistic discrete graphical models is usually conducted by varying one probability at a time and observing how this affects output probabilities of interest. When one probability is varied, then others are proportionally covaried to respect the sum-to-one condition of probability laws. The choice of  proportional covariation is justified by multiple optimality conditions, under which the original and the varied distributions are as close as possible under different measures. For variations of more than one parameter at a time, proportional covariation is justified in some special cases only.  In this work, for the large class of discrete statistical models entertaining a regular monomial parametrisation, we demonstrate the optimality of newly defined proportional multi-way schemes with respect to an optimality criterion based on the I-divergence. We demonstrate that there are varying parameters' choices for which proportional covariation is not optimal and identify the sub-family of distributions where the distance between the original distribution and the one where probabilities are covaried proportionally is minimum. This is shown by adopting a new geometric characterization of sensitivity analysis in monomial models, which include most probabilistic graphical models. We also demonstrate the optimality of proportional covariation for  multi-way analyses in Naive Bayes classifiers.
\end{abstract}

\begin{keyword}
Bayesian network classifiers \sep Covariation \sep I-projections \sep Monomial models \sep Sensitivity analysis. 



\end{keyword}

\end{frontmatter}



\section{Introduction}
\label{sec:introduction}

The effect of changes of a quantitative model's inputs to outputs of interest is studied through sensitivity analysis. It is a critical step for both the construction and the validation of a model as well as the communication of its results. Because of the importance of the application of sensitivity analyses, their use has been constantly increasing in the past ten years \citep{Ferretti2016}. The specific steps of a sensitivity analysis greatly vary depending on the quantitative model used \citep[see][for an overview]{Borgonovo2016,Saltelli2004}.

In the context of probabilistic graphical models, sensitivity analysis techniques are well-established \citep{Amsalu2017,Kaminski2018,Gaag2007,Yakaboski2018} and, especially for discrete Bayesian networks (BNs), applied in many real-world problems \cite{Chen2012,Hanninen2012,Kleemann2017, Makaba2020,Pitchforth2013}. A comprehensive review has been recently presented in \cite{Rohmer2020}. For such models, the validation process can be broken down into two steps: the first concerns the auditing of the validity of the conditional independences implied by the underlying graphical structure; the second, assuming the graph is valid, checks the impact of the numerical elicited probabilities on outputs of interest. Henceforth, the focus of this paper lies in this second phase.

The impact of the input probabilities in probabilistic graphical models is studied both at a local and a global level. At the local level, the effect of parameter variations on a specific output probability of interest is summarized by sensitivity functions \citep{Coupe2002,Gaag2007}. These correspond to the probability of interest written as a function of a varying parameter. At the global level, the effect of parameter variations are quantified by divergences and distances between the original and the varied model's distribution. Common measures to quantify the dissimilarity between the two distributions are the Chan-Darwiche distance and the Kullback-Leibler divergence \citep{Chan2005}.

Because of both its simplicity and its proven theoretical justifications, the most common investigation is the so-called \textit{one-way} sensitivity analysis, where the impacts of changes made to a single probability parameter are studied. One-way analyses are implemented in various pieces of software, e.g. SiamIam and the \texttt{bnmonitor} R package\footnote{\texttt{bnmonitor} is freely available for download at \url{https://github.com/manueleleonelli/bnmonitor}}. When one parameter is varied, then others are required to be adjusted, or \emph{covaried}, to respect the sum-to-one condition of probabilities. Although there are various ways to covary probabilities, the most common covariation scheme is the \emph{proportional} one, where, after a change to a parameter, the covarying parameters have the same proportion of the residual probability mass as they originally had \citep{Laskey1995,Renooij2014}. Proportional covariation in one-way analyses is \lq\lq{o}ptimal\rq\rq{}  since it minimizes a large array of divergences between the original and the varied probability distributions amongst any valid covariation scheme~\citep{Chan2005,Leonelli2017}.

Multi-way methods, where two or more parameters are varied contemporaneously, have not been extensively studied in the literature  \citep[see][for some exceptions]{Bolt2014,Bolt2017,Chan2004,Charitos2006,Bock2014,Kjaerulff2000,Renooij2012}. This is not only because they require a much more intensive computational power, but also, and more critically, because there are very little, if no, theoretical justifications in using a covariation scheme over another. Output probabilities have been shown to heavily depend on the covariation scheme used \cite{Leonelli2017,Renooij2014}. In \cite{Leonelli2017} it has been recently proven that for specific multi-way analyses called \emph{full single conditional probability table} (CPT) analyses, which highly restrict the parameters that can be varied, proportional covariation is optimal. However, to date there is no theoretical result guaranteeing the optimality of proportional covariation for multi-way variations.

This work provides a conclusive answer to the optimality of proportional covariation by characterizing the varying parameter choices in multi-way analyses for which proportional covariation minimizes the I-divergence, also known as Kullback-Leibler divergence~\citep{Csiszar2004}.  This is achieved via a new formal, geometric characterization of sensitivity analysis in terms of I-projections of the original distribution into a well-specified subset of the probability simplex. New flexible ways to choose varying parameters are introduced and their optimality with respect to the I-divergence is demonstrated.

As a consequence of these results,  the optimality of proportional covariation in naive Bayes classifiers (e.g.~\citep{Bielza2014}) for any combination of probabilities associated to feature variables is proven. Naive Bayes models are a specific type of BN classifiers often used to assign instances to a specific class in a classification problem. The tuning of the feature's probabilities is often critical to ensure that the classifier works reliably~\citep{Bolt2017}.

As in~\citep{Leonelli2017},  models where the probability of any element of the sample space is represented by a monomial are considered. They are called \emph{monomial models} (MMs). For ease of exposition, throughout the manuscript we use BNs as an illustration of MMs but there are many other well-known statistical models that can be seen as a specific instance of MMs \cite{Gorgen2015,Leonelli2019,Leonelli2017}: for instance staged trees and chain event graphs \cite{Smith2008}, context-specific BNs \cite{Boutilier1996}, decomposable Markov networks and probabilistic chain graphs~\citep{Koller2009}. 

The paper is structured as follows. Section~\ref{sec:monomial} includes the definition of MMs and some examples. Section~\ref{sec:proj} reviews the essential notions from information geometry. Section~\ref{sec:sensitivity}  summarizes the main steps of a sensitivity analysis. Section \ref{sec:multicov} introduces new classes of multi-way sensitivity analyses. Section~\ref{sec:geometric} gives a geometric characterization of sensitivity analysis and proves the optimality of our schemes.   Section~\ref{sec:real} discusses an applied sensitivity analysis in a medical application. The paper is concluded by a discussion. Longer proofs  are collated in \ref{sec:proof}. 

\section{Monomial discrete parametric models}
\label{sec:monomial}
Let $\mathbb Y$ be a finite set with $q$ elements  and 
 $\operatorname{P}$  a strictly positive probability density function for $\mathbb Y$.  
We write $\#\mathbb{Y}=q$, call $ y\in \mathbb Y$ an atom and $\operatorname{P}(y)$ the atomic probability of $ y$.  
The generic probability $\operatorname{P}$ can be seen as a point in the interior set of the $q$-dimensional simplex and we write $\operatorname{P}\in \Delta_{q-1}$.  Let also $[k]=\{1,2,\ldots,k\}$.
Next, to $\mathbb Y$ we associate a particular class of parametric statistical models, called monomial  models, in short MMs.

A MM is defined by three elements:
\begin{itemize}
\item a $q\times k$ matrix $A$ with non-negative integer entries, i.e. $A\in\mathcal{M}_{q\times k}(\mathbb{Z}_{\geq 0})$;
\item  a $k$-dimensional parameter vector $\theta$ with positive real entries, i.e. $\theta=(\theta_i)_{i\in[k]}\in\mathbb{R}^k_{>0}$;
\item a partition $S=\{S_1,\dots,S_n\}$ of $[k]$ such that $\theta_{S_i}\in\Delta_{\#S_i-1}$ for all $i\in[n]$.
\end{itemize}
 There is a row of $A$ for each atom $y$ and $A_y$ indicates the $y$-th row of $A$. 
The atomic probability of $y\in\mathbb{Y}$ given $\theta$ and $A$ is defined as  
$\operatorname{P}( y ) = \prod_{i \in [k]}  \theta_i^{ A_{ y,i}} = \theta^{A_y}$. In a MM  the parameters are grouped in such a way that those in a group sum to one. 
For a subset $S\subset[k]$, the notation $\theta_S=(\theta_i)_{i\in S}$ indicates the sub-vector of elements of $\theta$ indexed by $S$ and $\theta_S^{A_{y,S}}=\prod_{i\in S}\theta_i^{A_{y,i}}$ denotes the monomial associated to an event $y\in\mathbb{Y}$ where only parameters $\theta_i$ for $i\in S$ can have non-zero exponent.

\begin{definition} The MM  over $\mathbb Y$ associated to $\theta$, $A$ and $S$ is  defined as 
\begin{align*}
\operatorname{MM}(A,S) =
\left\{
\operatorname{P}  \in \Delta_{q-1} : \prod_{i\in[n]}\prod_{j\in S_i}\theta_j^{A_{y,j}}=\prod_{i\in[n]}\theta_{S_i}^{A_{y,S_i}} \mbox{ for }y\in\mathbb{Y}  \mbox{ and } \theta=(\theta_{S_i})_{i\in[n]}\in\bigtimes\limits_{i\in[n]}\Delta_{\#S_i-1}
\right\} 
\end{align*} 
\end{definition} 
The assumption of strictly positive probabilities is often met in practice, for instance for models learnt with complete data \citep{Heckerman1995}. The condition is imposed here to ensure the I-divergence exists and is finite. Degenerate cases are avoided by requiring $\theta\in\mathbb{R}^k_{>0}$ and in particular $A$ cannot have all elements of a row equal to one or all equal to zero. 

\begin{example}
The simplest example of a  $\operatorname{MM}(A,S)$ over a finite set $\mathbb Y$ is the saturated model where $\operatorname{P}(y)=\theta_y$ for all $y\in \mathbb Y$,  i.e. one parameter is associated to  the probability of each atomic event. In this case $\theta\in\Delta_{q-1}$ and $A$ is the $q$-by-$q$ identity matrix. 
\end{example}

\begin{example} \rm 
\label{ex:first}
 Let $\mathbb{Y}=[4]$, $A$ be the $4\times 3$ matrix with rows $A_1=(1,0,0)$, $A_2=(0,1,0)$, $A_3=(0,0,1)$ and $A_4=(0,1,1)$, and $S=[4]$, i.e. the sum of all parameters must be one. Atomic probabilities are defined by  $\operatorname{P}(y)=\theta_y$ for $y\in[3]$ and $\operatorname{P}(y)=\theta_2\theta_3$ for $y=4$, entailing 
 $\operatorname{P}(4)=\operatorname{P}(3)\operatorname{P}(2)$. This is not a MM since it is not possible to have $\sum_{i\in [4]}\operatorname{P}(i)=1$.
\end{example}

\begin{definition} \label{def:MDPMMM}
 A $\operatorname{MM}(A,S)$ is said to be \emph{multilinear} if $A\in\mathcal{M}_{q\times k}(\{0,1\})$.  A multilinear $\operatorname{MM}(A,S)$ is called \emph{regular} if for all $y\in\mathbb{Y}$ and all $i\in[n]$, $A_{y,j}=1$ for at most one $j\in S_i$ and zero otherwise.
\end{definition} 
A MM is multilinear if all its monomials are square-free, i.e. the exponents of the parameters are either zero or one. A regular MM is such that for each $y\in\mathbb{Y}$ and $S_i\in S$, there is at most one parameter $\theta_j$, $j\in S_i$, with non-zero exponent.

Henceforth we work with regular multilinear MMs. The authors have not been able to find a multilinear MM model which is not regular nor to prove that any multilinear square-free MM is regular. BNs, staged trees that admit a multilinear monomial representation as well as decomposable Markov networks and context specific BNs are regular. 

\subsection{Bayesian networks}
Many discrete statistical problems in a variety of domains are often modelled using BNs and there are now thousands of practical applications of these models~\citep{Bielza2014a,Cai2018,Drury2017,Mclachlan2020}. 
A BN expresses graphically a collection of conditional independences~\citep{Darwiche2009,Koller2009}. For a random vector $Y=(Y_i)_{i\in[m]}$ taking values in the Cartesian product $\mathbb{Y}=\bigtimes\limits_{i\in[m]}\mathbb{Y}_i$ and three disjoint subsets $B$, $C,$ and $D$ of $[m]$, the marginal vector $Y_B$ is said to be conditionally independent of $Y_C$ given $Y_D$ if 
$\operatorname{P}(Y_B=b|Y_C=c,Y_D=d)=\operatorname{P}(Y_B=b|Y_D=d)$,
for all $b\in\bigtimes\limits_{i\in B}\mathbb{Y}_i$, $c\in\bigtimes\limits_{i\in C}\mathbb{Y}_i$ and $d\in\bigtimes\limits_{i\in D}\mathbb{Y}_i$.  

A BN over a discrete random vector $Y=(Y_i)_{i\in[m]}$  is given by
\begin{itemize}
	\item $m-1$ \textit{conditional independence} statements of the form $Y_i\independent Y_{[i-1]}\;|\, Y_{\Pi_i}$, where $\Pi_i\subseteq [i-1]$;
	\item a \textit{directed acyclic graph} 
	 $\mathcal{G}=(\mathcal{V},\mathcal{E})$ with vertex set $\mathcal{V}=\{Y_i: i \in [m]\}$ 
	 and edge set $\mathcal{E}=\{(Y_i,Y_j):j\in[m],i\in\Pi_j\}$;
	\item conditional probabilities  $\operatorname{P}(Y_i=y_i|Y_{\Pi_i}=y_{\pi_i})$ for every $y_i\in\mathbb{Y}_i$, $y_{\pi_i}\in\bigtimes\limits_{j\in\Pi_i}\mathbb{Y}_j$ and $i\in[m]$. 
\end{itemize}

The components of the vector $Y_{\Pi_i}$  are said to be the \textit{parents} of the vertex $Y_i$,
and in the graphical representation of a BN there is an arrow from each component of $Y_{\Pi_i}$ pointing into $Y_i$.
For a vertex $Y_i$ with parents $Y_{\Pi_i}$, let $\theta_{y_iy_{\pi_i}}=\operatorname{P}(Y_i=y_i|Y_{\Pi_i}=y_{\pi_i})$.  The probability of any atom $y=(y_1,\dots,y_m)\in\mathbb{Y}$ can then be written as the monomial
$
\operatorname{P}(Y=y)=\prod_{i\in[m]}\theta_{y_iy_{\pi_i}}$ \citep{Castillo1997,Darwiche2003} Notice that $(\theta_{y_iy_{\pi_i}})_{y_i\in\mathbb{Y}_i}\in\Delta_{\#\mathbb{Y}_i-1}$ for all $i\in[m]$ and 
any possible value of the parent set of node $i$, namely $y_{\pi_i}\in\mathbb{Y}_{\Pi_i}$. 
 Thus a BN is a multilinear regular MM where parameters associated to each vertex conditionally to each combination of parents need to respect the sum-to-one condition.

\begin{example}
\label{ex:BN}
Suppose we are interested in studying how a population's health $(Y_3)$ is affected by both sports activity $(Y_1)$ and alcoholic drinking habits $(Y_2)$. Suppose these three variables can be categorized into high, medium and low, coded with $3$, $2$ and~$1$ respectively. Suppose that  health's levels are a function of both sports activity and drinking habits and that people who work out a lot tend to drink less alcohol. This situation can be depicted by a complete BN  with probabilities 
\[
\operatorname{P}(Y_1=i)=\theta_{i},\quad 
\operatorname{P}(Y_2=j|Y_1=i)=\theta_{ji},\quad 
\operatorname{P}(Y_3=l|Y_2=j,Y_1=i)=\theta_{lji}
\]
where $\sum_{k\in[3]}\theta_k=1$, $\sum_{k\in[3]}\theta_{ki}=1$ and $\sum_{k\in[3]}\theta_{kji}=1$ for all  $i,j\in[3]$. 
 The associated MM is given in Table~\ref{table:monomials} where the monomial representation of the $27$ atomic probabilities is listed. 
Of course the parameters $\theta_k$, $\theta_{ki}$ and $\theta_{kji}$ can be renamed to give some $(\theta_l)_{l\in[39]}$.
The $A$ matrix has dimension $27\times 39$ and is very sparse: in each row there is a one in three positions and zero otherwise.

\begin{table}
\begin{center}
\begin{tabular}{|ccccccccc|}
\hline
$\theta_{1}\theta_{11}\theta_{111}$&$\theta_{1}\theta_{11}\theta_{211}$&$\theta_{1}\theta_{11}\theta_{311}$&$\theta_{1}\theta_{21}\theta_{121}$&$\theta_{1}\theta_{21}\theta_{221}$&$\theta_{1}\theta_{21}\theta_{321}$&$\theta_{1}\theta_{31}\theta_{131}$&$\theta_{1}\theta_{31}\theta_{231}$&$\theta_{1}\theta_{31}\theta_{331}$\\
$\theta_{2}\theta_{12}\theta_{112}$&$\theta_{2}\theta_{12}\theta_{212}$&$\theta_{2}\theta_{12}\theta_{312}$&$\theta_{2}\theta_{22}\theta_{122}$&$\theta_{2}\theta_{22}\theta_{222}$&$\theta_{2}\theta_{22}\theta_{322}$&$\theta_{2}\theta_{32}\theta_{132}$&$\theta_{2}\theta_{32}\theta_{232}$&$\theta_{2}\theta_{32}\theta_{332}$\\
$\theta_{3}\theta_{13}\theta_{113}$&$\theta_{3}\theta_{13}\theta_{213}$&$\theta_{3}\theta_{13}\theta_{313}$&$\theta_{3}\theta_{23}\theta_{123}$&$\theta_{3}\theta_{23}\theta_{223}$&$\theta_{3}\theta_{23}\theta_{323}$&$\theta_{3}\theta_{33}\theta_{133}$&$\theta_{3}\theta_{33}\theta_{233}$&$\theta_{3}\theta_{33}\theta_{333}$\\
\hline
\end{tabular}
\end{center}
\caption{Monomial atomic probabilities for the BN of Example~\ref{ex:BN}.\label{table:monomials}}
\end{table}
\end{example}

\section{I-projections}
\label{sec:proj}
As a measure of closeness of two distributions we consider the I-divergence. Below we follow \cite[Chapter~3]{Csiszar2004}.

\begin{definition} 
\label{def:Idiv}

Let $\operatorname{P}$ and $\operatorname{Q}$ be two probability distributions over a finite space $\mathbb Y$. 
 The I-divergence (or Kullback-Leibler divergence) from P to Q  is defined as
\begin{align*}
\mathcal D(\operatorname{Q}|| \operatorname{P}) = \sum_{y\in \mathbb Y} \operatorname{Q}( y) \ln \frac{\operatorname{Q}( y)}{\operatorname{P}( y)}
\end{align*}
\end{definition}
It is often of interest to find the distribution that, within a given set, is closest to a given $\operatorname{P}$. 
\emph{I-projections} formalize this idea.  
 I-projections are used e.g. for maximum likelihood estimation in the context of exponential families.

\begin{definition}
Let $L$ 
 be a closed, convex set in the pointwise topology of distributions over $\mathbb Y$. The I-projection of a distribution $\operatorname{P}$  over $\mathbb Y$
 onto $L$ is  a  distribution $\operatorname{P}^ \ast \in L $ such that
\[
\mathcal D(\operatorname{P}^\ast || \operatorname{P}) = \min_{\operatorname{Q}\in L} \mathcal D(\operatorname{Q}||\operatorname{P}).
\] 
\end{definition} 
If $\operatorname{P}\in L$ then $\operatorname{P}^\ast=\operatorname{P}$. 
The fact that $L$ is closed and convex guarantees that $\operatorname{P}^\ast$ exists in $L$, and for strictly positive probabilities $\operatorname{P}^\ast$ is unique. 
\begin{theorem}
\label{theo:new}
Let $\operatorname{P}^*$ be the I-projection  of $\operatorname{P}$ in $L$. For all $\operatorname{Q}\in L$ it holds 
\[
 \mathcal D(\operatorname{Q} || \operatorname{P}) \geq \mathcal D(\operatorname{Q} || \operatorname{P}^\ast )  + \mathcal D(\operatorname{P}^\ast || \operatorname{P}).
\]
\end{theorem}
As a straightforward consequence of Theorem \ref{theo:new}, if the Pythagorean identity 
\begin{equation}
\label{eq:pitagora}
 \mathcal D(\operatorname{Q} || \operatorname{P}) = \mathcal D(\operatorname{Q} || \operatorname{R})  + \mathcal D(\operatorname{R} || \operatorname{P})  
 \end{equation}
holds  for all $\operatorname{Q}\in L$ and a specific $\operatorname{R}\in L$ then $\operatorname{R}=\operatorname{P}^\ast$.
Theorem \ref{theo:new} is used extensively in Section~\ref{sec:geometric} to prove the optimality of the new multi-way covariation schemes introduced  in Section~\ref{sec:multicov}.

\section{Sensitivity analysis}
\label{sec:sensitivity}

\subsection{Covariation}
\label{sec:cov}

When some parameters of a (conditional) probability distribution are varied to a new specific value,  then the remaining parameters need to be adjusted (or to \emph{covary}) to respect the sum-to-one condition of probability measures. In the binary case when one of the two parameters is varied this is straightforward, since the second parameter will be equal to one minus the other. But in generic discrete finite cases there are various considerations to be taken into account, as reviewed below.

We start by giving an alternative definition of a covariation scheme to~\citep{Renooij2014}. Our definition of covariation allows more than one or no parameters to be varied and maps into a probability simplex, i.e. the scheme is valid \citep{Renooij2014}. Let $k$ be the number of parameters in the model, 
$\emptyset$ the empty set and let $|v|$ denote the sum of the elements of a vector $v$.

\begin{definition}
For $\emptyset\neq V\subset S\subseteq[k]$, let $\theta_S \in \Delta_{\# S-1}$ be 
partitioned as $\theta_S=(\theta_V,\theta_{S\setminus V})$  and let  $\tilde{\theta}_V$ be such that $|\tilde{\theta}_V|\in(0,1)$. 
  A $\tilde\theta_V$-\emph{covariation scheme} is a function $\sigma$ from $\Delta_{\# S-1}$ to  $\Delta_{\# S-1}$ which fixes the subvector $\theta_V$ of $\theta_S$ to $\tilde\theta_V$, i.e.
\begin{align*}
\sigma: \hspace{0.5cm}  \Delta_{\# S-1}&\longrightarrow \Delta_{\# S-1}\\
(\theta_V,\theta_{S\setminus V})&\longmapsto (\tilde\theta_V,\cdot).
\end{align*}
When $V=\emptyset$, the $\tilde\theta_V$-covariation scheme is the identity function. 
\label{def:covar}
\end{definition}

Thus $\theta_S$ denotes a vector of parameters that need to respect the sum to one condition,  $\tilde\theta_V$ denotes the new numerical specification of the parameters varied, i.e. those with index in a set $V$, and the values of the parameters with index in $[k]\setminus S$ do not vary.
Below we generalise some frequently applied  covariation schemes. 
\begin{definition}
\label{def:schemes}
In the notation of Definition~\ref{def:covar}
\begin{itemize}
\item the $\tilde\theta_V$-\textit{proportional} covariation scheme $\sigma_{\operatorname{pro}}(\theta_S)=(\tilde\theta_V,\tilde\theta_{S\setminus V})$ is defined by setting 
	\[
		\tilde\theta_j=
		\frac{1-|\tilde{\theta}_V|}{1-|\theta_V|}\theta_j \qquad  \text{for all } j\in S\setminus V. 
	\] 
\item The $\tilde\theta_V$-\textit{uniform} covariation scheme, $\sigma_{\operatorname{uni}}(\theta_S)=(\tilde\theta_V,\tilde\theta_{S\setminus V})$ is defined by setting
	\[
	\tilde\theta_j=\frac{1-|\tilde{\theta}_V|}{\#S-\#V}  \qquad  \text{for all } j\in S\setminus V. 
	\]
	\end{itemize}
\end{definition}

Different covariation schemes may entertain different properties which, depending on the domain of application, might be more or less desirable~\citep[see][for a list]{Leonelli2017,Renooij2014}. Definition~\ref{def:schemes} extends the proportional and uniform covariation  schemes given in \citep{Renooij2014} to cases where one or more parameters are varied.

\begin{figure}
\begin{center}
\scalebox{0.8}{
\begin{tikzpicture}
\begin{axis}[
  view/h=134.5,
  axis lines=center,
  xmax=1.5, 
  ymax=1.5,
  zmax=1.5,
  ytick={1},
]
\addplot3[patch,red!70!black!50,forget plot] 
  coordinates 
  {
  (1,0,0) 
  (0,1,0) 
  (0,0,1)
  };
\addplot3[no markers,black,line width=1pt] 
  coordinates 
  { 
  (0.4,0,0.6) 
  (0.4,0.1,0.5) 
  (0.4,0.2,0.4) 
  (0.4,0.3,0.3) 
  (0.4,0.4,0.2) 
  (0.4,0.5,0.1) 
  (0.4,0.6,0) 
  };
\addplot3[no markers,black,dotted,line width=1.5pt] 
  coordinates 
  { 
  (0.6,0.3,0.1) 
  (0.4,0.3,0.3)
  };
  \addplot3[no markers,black,dotted, line width=1.5pt] 
  coordinates 
  { 
  (0.1,0.2,0.7) 
  (0.4,0.3,0.3)
  };
    \addplot3[no markers,black, dashed,line width=1.5pt] 
  coordinates 
  { 
  (0.1,0.2,0.7) 
  (0.4,0.133,0.467)
  };
      \addplot3[no markers,black,dashed,line width=1.5pt] 
  coordinates 
  { 
  (0.6,0.3,0.1) 
  (0.4,0.45,0.15)
  };
\node[fill=black,inner sep=1pt,circle,label={180:$\theta_{S1}$}] 
  at (axis cs:0.6,0.3,0.1) {};  
\node[fill=black,inner sep=1pt,circle,label={180:$\theta_{S2}$}] 
  at (axis cs:0.1,0.2,0.7) {};  
  \node[fill=black,inner sep=1pt,rectangle] at (axis cs:0.4,0.3,0.3){};
   \node[fill=black,inner sep=1pt,rectangle] at (axis cs:0.4,0.45,0.15){};
      \node[fill=black,inner sep=1pt,rectangle] at (axis cs:0.4,0.133,0.467){};
\end{axis}
\end{tikzpicture}}
\end{center}
\caption{Graphical representation of uniform and proportional covariation in Example~\ref{ex:ciao} for  $\theta_{S1},\theta_{S2}\in\Delta_2$.\label{fig:cov}}
\end{figure}

\begin{example}
\label{ex:ciao}
Consider  $\theta_S=(\theta_1,\theta_2,\theta_3)\in\Delta_2$, 
$V=\{1\}$ and $\tilde\theta_V=0.4$. The simplex $\Delta_2$ is given by the  surface in Figure~\ref{fig:cov}, whilst the dark full line denotes the  image of  any $\tilde\theta_V$-covariation scheme $\sigma$ which fixes $\tilde\theta_1=0.4$, i.e. the  set defined by the intersection of the simplex with the line defined by $\theta_1= 0.4$. That is, this line  describes all possible ways $\theta_2$ and $\theta_3$ can be covaried. When $\sigma$ is the uniform covariation scheme any $\theta_S\in\Delta_2$ is projected to the same point, as illustrated by the dotted lines in Figure~\ref{fig:cov}. 
Conversely, the dashed lines  refer to the proportional covariation scheme  which can  project points $\theta_S\in\Delta_2$ to different elements. 
\end{example}

The order-preserving covariation scheme in Definition~\ref{def:covarOrder} follows from changing one parameter  and its definition is slightly convoluted.
 Let thus $V=\{v\}$ have just one element and consider $S \supset V$ such that $|\theta_S|=1$.
  Assume that  $\theta_v$ is not the largest component of $\theta_S$ 
  and order the components of $\theta_S$ from the smallest to the largest. 
  Without loss of generality by reorganising the indices in $\theta$, 
  we can assume that the ordered components of $\theta_S$ are $\theta_1\leq \cdots \leq \theta_v<\cdots \leq \theta_{\#S}$.

\begin{definition} \label{def:covarOrder}
In the notation of Definition~\ref{def:covar}
 the $\tilde\theta_V$-\textit{order preserving} covariation scheme 
 is defined, according to whether $\theta_v$ is increased or decreased, by setting 
	\[
	\tilde\theta_j=\left\{
	\begin{array}{ll}
	\displaystyle \frac{ \tilde{\theta}_v}{\theta_v} \theta_j  , 
			&\mbox{if } j<v \mbox{ and } \tilde{\theta}_v\leq \theta_v,\\
	-\theta_j \displaystyle \frac{ 1-\theta_{\operatorname{suc}}}{\theta_{\operatorname{suc}}} \displaystyle\frac{\tilde{\theta}_v}{\theta_v}
		+\displaystyle \frac{\theta_j}{\theta_{\operatorname{suc}}}, 
			&\mbox{if } j>v \mbox{ and } \tilde{\theta}_v\leq \theta_v,\\
	\displaystyle \theta_j \frac{ \theta_{\operatorname{max}} -\tilde{\theta}_v }{\theta_{\operatorname{max}}-{\theta}_v} , 
			&\mbox{if } j<v \mbox{ and } \tilde{\theta}_v> \theta_v,\\
	\displaystyle (\theta_{\operatorname{max}} -\tilde{\theta}_v) \frac{{\theta}_j-\theta_{\operatorname{max}}}{\theta_{\operatorname{max}}-{\theta}_v} 
		+\theta_{\operatorname{max}},
			&\mbox{if } j>v \mbox{ and } \tilde{\theta}_v> \theta_v 
	\end{array} 
	\right .  
\]
where $\theta_{\operatorname{max}}=1/(1+\#S-v)$ is the upper bound for $\tilde{\theta}_v$ and $\theta_{\operatorname{suc}}=\sum_{k=v+1}^{\#S}\theta_k$ is the original total mass of the parameters in $\theta_S$ larger than $\theta_v$.
\end{definition}

Definitions~\ref{def:covar} to~\ref{def:covarOrder} assume $\theta_S\in\Delta_{\#S-1}$. Next we specialise them to apply to the parameter vector $\theta$ of a MM.

\begin{definition}
\label{def:bellissima}
In the notation of Definitions~\ref{def:MDPMMM} and~\ref{def:covar}, 
let $\theta$ be the parameter vector of a MM. For $V\subset[k]$, let $V_i=S_i\cap V$ and $\sigma_i$ a $\tilde\theta_{V_i}$-covariation scheme for each $i \in [n]$. 
Then:
\begin{itemize}
\item a $\tilde{\theta}_V$-\emph{covariation scheme} for $\theta$ is a function $\sigma:\bigtimes\limits_{i\in[n]}\Delta_{\#S_i-1}\rightarrow\bigtimes\limits_{i\in[n]}\Delta_{\#S_i-1}$  such that $\sigma_{|S_i}=\sigma_i$, where $\sigma_{|S_i}$ denotes the restriction of $\sigma$ over $\Delta_{\#S_i-1}$.
\item a $\tilde{\theta}_V$-covariation scheme for $\theta$ is called \textit{proportional} if $\sigma_i$ is a $\tilde\theta_{V_i}$-proportional covariation scheme whenever $V_i\neq\emptyset$.
\end{itemize}
\end{definition}

Definition~\ref{def:bellissima} formalizes how parameters in a MM need to covary for any choice of varied parameters  $\tilde\theta_V$. For instance, in a BN  model the sets $S_i$, 
 $i \in [n]$, denote the conditional probability distributions of any vertex given a specific combination of parents. If a full single CPT analysis is performed then $\tilde\theta_V$ includes one parameter from each conditional distribution associated to a given vertex. For such distributions, since $V_i$ is non-empty, a standard $\sigma_i$ covariation scheme is applied. In all other cases $V_i$ is empty and the 
$\tilde{\theta}_V$-covariation scheme for $\theta$ returns the original value of the parameters since $\sigma_i$ is defined as the identity function.

Notice that our definition of a covariation scheme for a MM model encompasses all types of sensitivity analyses usually considered in BN models: one-way sensitivity analysis if $V$ consists of one element only, full single CPT analyses as illustrated in the previous paragraph, multi-way analyses where multiple parameters from the same conditional distribution are varied, or generic multi-way analyses where any combination of parameters can be varied. 

Our definition of a $\tilde\theta_{V_i}$-covariation scheme $\sigma$ guarantees that if $\operatorname{P} \in \operatorname{MM}(A,S)$ then  $\sigma(\textnormal{P}) \in \operatorname{MM}(A,S)$ for all covariation schemes $\sigma$, i.e. $\sigma(\textnormal{P})$ belongs to the same monomial model. 

\subsection{Sensitivity functions}
Sensitivity functions~\cite{Chan2002,Coupe2002} 
are frequently used during  model validation to investigate how an output probability of interest varies as one (or possibly more) model's parameter is allowed to change.
They are particularly useful since, for instance, the conditional specification of probabilities in a BN might imply a marginal probability which appears to be unreasonable to a user, although being a coherent consequence of his/her beliefs. Sensitivity functions  depict the required change of a  parameter that would give a reasonable marginal probability. 
In essence they represent the functional relationship between a parameter being varied and the output probability of an event of interest.

Let $\operatorname{MM}(A,S)$ be a multilinear model over a finite set $\mathbb Y$
 whose parameter $\theta$ is a $k$-dimensional vector partitioned in $S_1,\ldots, S_n$ as in Definition~\ref{def:MDPMMM}. 
In the setting of Definition~\ref{def:bellissima} we fix $V\subset [k]$ and consider $V_i=V\cap S_i$ for $i\in [n]$. 
Now we  allow the $  \tilde\theta_{V_i}$, $i\in [n]$, to vary so that $| \tilde\theta_{V_i} | \in (0,1)$  and  for each value of $\tilde\theta_V$
 consider a  $\tilde\theta_V$-covariation scheme. 
 For this family of covariation schemes and for an event $E\subset \mathbb Y$ of interest,
 Definition~\ref{def:sensitivityfunction} gives the probability of $E$ under the different covariation schemes in the family.  
\begin{definition} \label{def:sensitivityfunction} 
Let $\sigma$ be a $\tilde\theta_V$-covariation scheme.  
For $\operatorname{P} \in \operatorname{MM}(A,S)$ 
the probability 
$\sigma(\operatorname{P})( E)$ read
as function of $\tilde\theta_V$ is called sensitivity function associated to the (family of) $\tilde\theta_V$-covariation schemes. 
\end{definition}
%
The fact that the resulting probability of an event of interest $\sigma(\operatorname{P})(E)$ depends on the covariation scheme used is illustrated by the following example. 

\begin{table}
\begin{center}
\begin{tabular}{|cccccc|}
\hline
$\theta_{1}=0.2$&$\theta_{2}=0.3$&$\theta_{3}=0.5$&
$\theta_{11}=0.2$&$\theta_{21}=0.3$&$\theta_{31}=0.5$\\
$\theta_{12}=0.3$&$\theta_{22}=0.3$&$\theta_{32}=0.4$&
$\theta_{13}=0.7$&$\theta_{23}=0.2$&$\theta_{33}=0.1$\\
$\theta_{111}=0.1$&$\theta_{211}=0.2$&$\theta_{311}=0.7$&
$\theta_{112}=0.1$&$\theta_{212}=0.3$&$\theta_{312}=0.6$\\
$\theta_{113}=0.2$&$\theta_{213}=0.3$&$\theta_{313}=0.5$&
$\theta_{121}=0.1$&$\theta_{221}=0.4$&$\theta_{321}=0.5$\\
$\theta_{122}=0.3$&$\theta_{222}=0.6$&$\theta_{322}=0.1$&
$\theta_{123}=0.3$&$\theta_{223}=0.5$&$\theta_{323}=0.2$\\
$\theta_{131}=0.8$&$\theta_{231}=0.1$&$\theta_{331}=0.1$&
$\theta_{132}=0.7$&$\theta_{232}=0.2$&$\theta_{332}=0.1$\\
$\theta_{133}=0.4$&$\theta_{233}=0.5$&$\theta_{333}=0.1$&
&&\\
\hline
\end{tabular}
\end{center}\caption{Probabilities associated to the BN in Example~\ref{ex:BN}. \label{table:app1}}
\end{table}

\begin{figure}
\begin{center}
\includegraphics[scale=0.6]{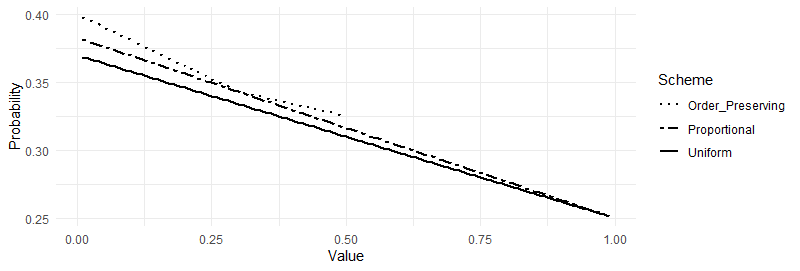}
\end{center}
\vspace{-0.5cm}
\caption{Sensitivity functions of Example~\ref{ex:2} for different covariation schemes.\label{fig:sens}}
\end{figure}

\begin{example}
\label{ex:2}
The BN of Example~\ref{ex:BN} is refined with the numerical specification of its parameters given in Table~\ref{table:app1}.
The probability of an individual being healthy is the event of interest, i.e. $\operatorname{P}(Y_3=3)$. 
Given the elicited probabilities, this is equal to $0.343$. However, for the population investigated $\operatorname{P}(Y_3=3)$ is known to be lower than $0.3$. 
To achieve this upper bound it is decided to try varying the parameter $\theta_2 \in (0,1)$.  
Figure~\ref{fig:sens}  reports the sensitivity functions for three families of covariation schemes. 
In a family all $\tilde\theta_2$-covariation schemes are proportional (dotted/dashed line), in another family they are uniform (solid line) and in the third family they are order-preserving (dotted line).
For proportional and uniform covariation, 
$\theta_2$ needs to be varied to around $0.6$, whilst for order-preserving covariation the required bound cannot be achieved, indeed order-preserving covariation restricts the values the varied parameter can take. 
\end{example}

As highlighted by Example~\ref{ex:2} a parameter variation might be enforced to entertain some specific bounds on probabilities of interest. However these  probabilities are affected by the choice of the covariation scheme. In some simple cases proportional covariation has been demonstrated to be optimal, in the sense that the original and the resulting probability distributions are as close as possible. 

\subsection{Global dissimilarity}
The closeness of the original and varied distributions can be quantified using different distances or divergences. 
The most commonly used distance in sensitivity studies is the so called CD distance~\citep{Chan2005}.
The CD distance is defined as the DeRobertis distance, which has been used for quite some time in the Bayesian inference literature~\citep{Gustafson1995}.
For two probability distributions $\operatorname{P}$ and $\operatorname{Q}$ over a finite space $\mathbb{Y}$ this is 
\[
\mathcal{D}_{\operatorname{CD}}(\operatorname{P},\operatorname{Q})=\ln\max_{y\in\mathbb{Y}}\left(\frac{\operatorname{P}(y)}{\operatorname{Q}(y)}\right)-\ln\min_{y\in\mathbb{Y}}\left(\frac{\operatorname{P}(y)}{\operatorname{Q}(y)}\right)=\max_{y,y'\in\mathbb{Y}}\ln \left(\frac{\operatorname{P}(y)\operatorname{Q}(y')}{\operatorname{P}(y')\operatorname{Q}(y)}\right).
\]
Until recently, proportional covariation had a theoretical justification only for one-way analyses in BN models, since this scheme minimizes the CD distance between the original and the varied distributions~\citep{Chan2005}. In~\citep{Leonelli2017} it is proven that this is also true for full single CPT analyses in any multilinear MM. 

Proportional covariation also minimizes the $\phi$-divergence from the original to the varied distribution in  full single CPT analyses \citep{Leonelli2017}. The $\phi$-divergence  from $\operatorname{P}$ to $\operatorname{Q}$ is defined as
\[
\mathcal{D}_{\phi}(\operatorname{Q}||\operatorname{P})=\sum_{y\in\mathbb{Y}}\operatorname{P}(y)\phi\left(\frac{\operatorname{Q}(y)}{\operatorname{P}(y)}\right), \hspace{1cm} \phi\in\Phi,
\]
where $\Phi$ is the class of convex functions $\phi(x)$, $x\geq 0$, such that $\phi(1)=0$, $0 \, \phi(0/0)=0$ and $0\phi(x/0)=\lim_{x\rightarrow \infty}\phi(x)/x$. The I-divergence in Definition~\ref{def:Idiv} can be seen as a special instance of $\phi$-divergences for $\phi(x)=x\ln(x)$. 

\begin{figure}
    \centering
\includegraphics[scale=0.6]{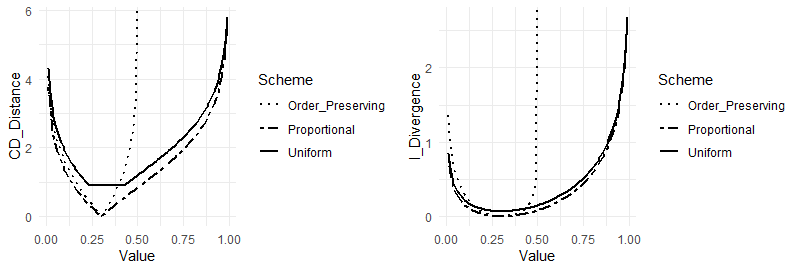}
    \caption{Measures of dissimilarity for the BN in Example~\ref{ex:BN} under different covariation schemes: CD distance (left) and I-divergence (right).\label{fig:diss}}
    \end{figure}
    
\begin{example}
Figure~\ref{fig:diss} reports the CD distance and I-divergence for the BN in Example~\ref{ex:BN} and for variations of $\theta_2$ under different covariation schemes. These plots show the optimality of proportional covariation (dashed/dotted line) which for both metrics takes smaller values than the other schemes. 
\end{example}

\section{Novel multi-way sensitivity analyses}
\label{sec:multicov}
Full single CPT analyses, for which the optimality of proportional covariation has been already proven, highly restrict the parameters that can be varied. To our knowledge, the only  attempt in defining other multi-way sensitivity analyses is given in~\citep{Bolt2017}, 
where \emph{balanced} analyses are introduced. In a nutshell, these reduce a multi-way problem into a one-way analysis by restricting the possible parameter variations. 

Because of the monomial structure of their atomic probabilities, more general ways to choose parameters to be varied can be defined for MMs.
These new multi-way analyses depend on the partition $\{S_1,\dots,S_n\} $  of the $k$ parameters of a MM and on the sets $V_i=V\cap S_i$, where $V\subset[k]$ and $i\in[n]$. 
Let $C=\bigcup\limits_{i\in[n]: V_i\neq\emptyset}S_i$ be the union of all $S_i$ for which $V_i$ is not empty and $F=[k]\setminus C$. 
The set $F$ includes the indices of the parameters that do not need to be covaried, whilst $C$ is the index set of the (co)varied parameters. Definition \ref{def:multi} gives special ways to choose $V$ which depend on the model structure.

\begin{definition}
\label{def:multi}
A sensitivity analysis  is said to be  
\begin{itemize}
\item \emph{simple}  if, for all $i,j\in C$, $\theta_i\theta_j$  does not divide $\theta^{A_y}$ for  any $y\in\mathbb{Y}$
	\item 
\emph{complete}  if there exists at least one $y\in\mathbb{Y}$ such that $\theta_H$ divides $\theta^{A_y}$ for all  $H\in\times_{i\in [n]:V_i\neq\emptyset}S_i$ 
\item \emph{ordered} if, given an ordered sequence of sets $S_{k_1},\dots, S_{k_l}$ such that $V_{k_i}=V\cap S_{k_i}\neq\emptyset$, $i\in[l]$, and $\{S_{k_1},\dots, S_{k_l}\}\subseteq \{S_1,\dots,S_n\}$,  there exists at least one $y\in\mathbb{Y}$ such that $\theta_H$ divides $\theta^{A_y}$ for all $H$ in the set 
\begin{equation}
\label{eq:nuova}
 \left\{ S_{k_1}\setminus V\right\} \cup_{i\in[l-1]}\left\{ \times_{j\in [i]}V_{k_j}\times \{S_{k_{i+1}}\setminus V\} \right\} \cup \{\times_{i\in[l]}V_{k_i}\}.
\end{equation}
\end{itemize}
\end{definition}

The set $C$ in a simple sensitivity analysis includes the indices of the varied or covaried parameters. In a complete  analysis all sets $H\in\times_{i\in [n]:V_i\neq\emptyset}S_i$ include the same number of indices, equal to $\sum_{i\in[n]}\mathds{1}_{\{S_i\cap V\neq \emptyset\}}$. Conversely  in an ordered analysis different sets $H$ include a different number of indices. This implies that all monomials $\theta_H$ have the same degree in complete analyses, whilst they have different degrees in ordered ones.

\begin{example}
For an ordered analysis suppose $\{S_{k_1},\dots,S_{k_l}\}=\{S_1,S_2,S_3\}$. Then (\ref{eq:nuova}) becomes
\[
\{S_1\setminus V\}\cup \{V_1\times\{S_2\setminus V\}\}\cup \{V_1\times V_2\times\{S_3\setminus V\}\}\cup \{V_1\times V_2\times V_3\}.
\]
\end{example}

Although the definition of such new analyses may appear obscure, these have a very simple graphical interpretation and include some well-known sensitivity analyses. 
In a simple sensitivity analysis no (co)varied parameters appear in the same monomial. It thus includes the following analyses in BN models:
\begin{itemize}
\item one-way sensitivity analyses: one parameter of a CPT of a vertex is varied
\item full single CPT analyses: one parameter from each conditional distribution of a vertex is varied
\item multi-way analyses where two or more parameters from one conditional distribution are varied
\item multi-way analyses where parameters from CPTs associated to incompatible parent configurations are varied. To see this consider the BN of Example \ref{ex:BN} and suppose the parameters $\theta_{22}$ and $\theta_{311}$ are varied, implying respectively that $Y_1=2$ and $Y_1=1$. From Table~\ref{table:monomials} we can see that no two parameters from the associated conditional probability laws appear in the same monomial
\item any combination of the four  above.
\end{itemize}

It is straightforward to see that no two (co)varied  parameters appears in the same monomial representing the probabilities of a BN for all choices listed above.

Complete sensitivity analyses are such that all possible combinations of (co)varied parameters in different sets $S_i$, for $i\in [n]$ such that $V_i\neq\emptyset$, appear in at least one monomial. The simplest possible example of such analyses is in the case of a BN consisting of two independent random variables where one parameter from each distribution is varied. But more generally such analyses are associated to varied parameters in CPTs implying disjoint parent sets. To illustrate this consider the BN in Figure~\ref{fig:extot}. The variation of one parameter from the distribution of $Y_3|Y_1$ and another from the distribution of $Y_5|Y_2$ would give a complete  sensitivity analysis since the conditioning variables are different. As an additional example, consider the complete BN of Example \ref{ex:BN} and suppose $\theta_1$ and $\theta_{122}$ are varied. Notice that $(\theta_1,\theta_2,\theta_3)\in\Delta_2$ and $(\theta_{122},\theta_{222},\theta_{322})\in\Delta_2$. This choice of parameters would be a complete analysis if every element of 
\[
(\theta_1\theta_{122},\theta_{1}\theta_{222},\theta_{1}\theta_{322},\theta_2\theta_{122},\theta_{2}\theta_{222},\theta_{2}\theta_{322},\theta_3\theta_{122},\theta_{3}\theta_{222},\theta_{3}\theta_{322})
\]
appeared in at least one monomial in Table \ref{table:monomials}. Since this is not the case, this choice of parameters does not correspond to a complete sensitivity analysis. More generally, it can be seen that for any choice of two parameters each from a different conditional distribution (possibly in the same CPT) there is no pair of $\theta$'s such that all combinations of covaried parameters appear in the same monomial.

\begin{figure}
\entrymodifiers={++[o][F-]}
\centerline{
\xymatrix{
Y_3&Y_4&Y_6\\
Y_1\ar[u]\ar[ur]&Y_2\ar[u]\ar[ru]\ar[r]&Y_5\ar[u]
}
}
\caption{A BN to illustrate fully dependent and conditionally dependent sensitivity analyses. \label{fig:extot}}
\end{figure}

Ordered analyses imply an order over the varied parameters.  A varying parameter needs to be a probability which is conditional on the events associated to preceding varying parameters in this order. 
An example from Figure~\ref{fig:extot} illustrates this for BNs. Suppose the parameter associated to $\operatorname{P}(Y_2=y_2)$ is varied. Then in an ordered analysis any parameter from $\operatorname{P}(Y_5=y_5|Y_2=y_2)$ can be varied and, if so, also any parameter from $\operatorname{P}(Y_6=y_6|Y_5=y_5,Y_2=y_2)$. As an additional example, consider again the complete BN in Example \ref{ex:BN} and suppose the parameters $\theta_1$ and $\theta_{311}$ are varied, defining the set $V$ of varied indexes. Let $\theta_{S_1}=(\theta_1,\theta_2,\theta_3)$ and $\theta_{S_2}=(\theta_{111},\theta_{211},\theta_{311})$. For this choice of parameters  (\ref{eq:nuova}) can be written as $S_1\setminus V\cup \{V_1\times \{S_2\setminus V\}\}\cup \{V_1\times V_2\}$, where $V_i=S_i\cap V$, for $i=1,2$. The monomials indexed by $S_1\setminus V$ are $(\theta_2,\theta_3)$, those indexed by $V_1\times\{V_1\times\{S_2\setminus V\}\}$ are $(\theta_1\theta_{111},\theta_1\theta_{211})$ and those indexed by $V_1\times V_2$ are $(\theta_1\theta_{311})$. This choice of parameters corresponds to an ordered analysis since any of these monomials divides at least one monomial atomic probability in Table \ref{table:monomials}. Conversely, the choice of varied parameters $\theta_1$ and $\theta_{122}$, for instance, would not correspond to an ordered sensitivity analysis.

\section{The optimality of proportional covariation} 
\label{sec:geometric} 

After the variation of some parameters of a MM with indices in $V$ to a value $\tilde\theta_V$,  a $\tilde\theta_V$-covariation scheme needs to be applied to respect all sum-to-one conditions of the model. In this section, given $\operatorname{P}\in \textnormal{MM}(A,S)$ and its $\tilde\theta_V$-proportional covariation $\tilde{\operatorname{P}}$ we first determine a family $L$ of $Q$ densities for which the Pythagorean equality 
$
\mathcal{D}(\operatorname{Q} || \operatorname{P})=\mathcal{D}(\operatorname{Q} || \tilde{\operatorname{P}})+\mathcal{D}(\tilde{\operatorname{P}} || \operatorname{Q})
$ 
holds. This will later allow us to demonstrate for which cases a $\tilde\theta_V$-proportional covariation scheme is optimal.

\subsection{A geometric characterization of sensitivity analysis}

In the notation of Definition \ref{def:bellissima}, let $\emptyset\neq V\subset [k]$, $C=\bigcup\limits_{i\in[n]:V_i\neq\emptyset}S_i$ and $F=[k]\setminus C$. Let $\Delta^F=\bigtimes\limits_{i\in[n]:V_i=\emptyset}\Delta_{\#S_i-1}$ and $\Delta^C=\bigtimes\limits_{i\in[n]:V_i\neq\emptyset}\Delta_{\#S_i-1}$. The set $[k]$ is so partitioned into $V$, $C\setminus V$ and $F$, namely the index set of the varied, covaried and fixed parameters. A generic parameter vector can be written as $\theta=(\theta_F,\theta_{V},\theta_{C\setminus V})$ and for $\operatorname{P}\in \textnormal{MM}(A,S)$ the atomic probability of $y\in\mathbb{Y}$ can be written as $\operatorname{P}(y)=\theta_F^{A_{y,F}}\theta_V^{A_{y,V}}\theta_{C\setminus V}^{A_{y,C\setminus V}}$. For any given $\theta_F\in\Delta^F$, $\textnormal{Slice}(\theta_F)$ is the subset of densities in $\textnormal{MM}(A,S)$ for which the parameters indexed by $F$ take value $\theta_F$, namely
\[
 \operatorname{Slice}(\theta_F) 
= \left\{\operatorname{P}  \in \textnormal{MM}(A,S):\operatorname{P}(y)=\theta_F^{A_{y,F}}\theta_C^{A_{y,C}}  \text{ for all }\theta_C\in\Delta^C \text{ and } y\in\mathbb{Y}\right\},
\]
where $\theta_C^{A_{y,C}}=\theta_V^{A_{y,V}}\theta_{C\setminus V}^{A_{y,C\setminus V}}$. It holds
$
\operatorname{MM}(A,S)=\bigcup\limits_{\theta_F\in\Delta^F}\operatorname{Slice}(\theta_F)
$.

\begin{example}
Consider the complete BN in Example \ref{ex:BN} with monomial atomic probabilities given in Table \ref{table:monomials}. Suppose the parameters $\theta_1$, $\theta_{11}$ and $\theta_{111}$ are varied. Then $\theta_V=(\theta_1,\theta_{11},\theta_{111})$, $\theta_{C\setminus V}=(\theta_2,\theta_3,\theta_{21},\theta_{31},\theta_{211},\theta_{311})$ and 
\[
\theta_F=(\theta_{i2},\theta_{i3},\theta_{i21},\theta_{i31},\theta_{i12},\theta_{i22},\theta_{i23},\theta_{i13},\theta_{i23},\theta_{i33})_{i\in[3]}.
\]
In $\operatorname{Slice}(\theta_F)$ the entries in $\theta_F$ are fixed to the values given in Table \ref{table:app1} and are not allowed to vary, whilst those in $\theta_V$ and $\theta_{C\setminus V}$ can vary. Atomic probabilities in $\operatorname{Slice}(\theta_F)$ are still written as $\operatorname{P}(y)=\theta_i\theta_{ji}\theta_{lji}$, for appropriate choices of the indexes $i$, $j$ and $l$.
\label{ex:sub}
\end{example}

  Theorem \ref{theo:pinco} shows that $\operatorname{P}\in\textnormal{MM}(A,S)$ and its $\tilde\theta_V$-proportional covariation density belong to the same slice.

\begin{theorem}
\label{theo:pinco}
Let $\theta=(\theta_F,\theta_V,\theta_{C\setminus V})$ be the parameter vector of $\operatorname{P}\in \textnormal{MM}(A,S)$ and let $\tilde\theta=(\tilde\theta_F,\tilde\theta_V,\tilde\theta_{C\setminus V})$ be the parameter vector of the $\tilde\theta_V$-proportional covariation of $\operatorname{P}$ called $\tilde{\operatorname{P}}$. Then $\tilde{\operatorname{P}}\in\textnormal{Slice}(\theta_F)$, that is $\tilde\theta_F=\theta_F$.
\end{theorem}
\proof
The proof follows straightforward from Definition \ref{def:bellissima}. Indeed $\tilde\theta_F=\theta_F$, $\tilde\theta_V$ is given and $
\tilde\theta_{C\setminus V}=\left(\left(\frac{1-|\tilde\theta_{V_i}|}{1-|\theta_{V_i}|}\theta_j\right)_{j\in S_i}\right)_{i\in[n]}
$.
\endproof

Theorem \ref{theo:pinco} guarantees that proportional covariation schemes in MMs does not affect probabilities that customarily are not changed in sensitivity analysis, i.e. those in  Slice$(\theta_F)$.
Next, let 
\[
L_{\textnormal{sensi}}=\textnormal{Slice}(\theta_F)\cap \{P\in \textnormal{MM}(A,S): \theta_V=\tilde\theta_V\},
\]
denote the family of distributions where only the parameters $\theta_{C\setminus V}$ can vary. It follows from Theorem \ref{theo:pinco} that $\tilde{\operatorname{P}}\in L_{\textnormal{sensi}}$.

\begin{example}
In the setup of Example \ref{ex:sub}, suppose that the parameters in $\theta_V$ are varied to a new value $\tilde\theta_V$. Then, in $L_{\textnormal{sensi}}$ there are those distributions with atomic probabilities in Table \ref{table:monomials}, where $\theta_F$ are held fixed and $\theta_V$ are fixed to a new value $\tilde\theta_V$: only the parameters $\theta_{C\setminus V}$ can vary.
\end{example}

\begin{example}
Consider the setup of Example \ref{ex:ciao} where $\theta=(\theta_1,\theta_2,\theta_3)\in\Delta_2$, $V=1$ and $\tilde\theta_V=0.4$. Since this situation corresponds to a simple discrete probability distribution, Slice$(\theta_F)=\Delta_2$, i.e. there are no parameters in $\theta_F$. So Slice$(\theta_F)$ is the simplex represented in Figure \ref{fig:cov}. The space $L_{\textnormal{sensi}}$ is represented by the full line traversing the simplex, corresponding to the intersection of $\tilde\theta_1 = 0.4$ with $\Delta_2$.
\end{example}

The two examples above highlighted that $L_{\textnormal{sensi}}$ includes the model's distributions where only the parameters in $C\setminus V$, those that need to be covaried, can vary. Henceforth, we thus consider distributions in $L_{\textnormal{sensi}}$.

\subsection{A comprehensive example}
\label{sec:basta}
The following example demonstrates that given a $\operatorname{MM}(A,S)$ the choice of parameters varied affects the form of the family of densities for which the Pythagorean equality holds. Consider the following setup: a student can either fail (coded as 1), pass (coded as 2) or obtain a distinction (coded as 3) in an exam. If the student fails, she is given an additional chance where again she can either fail, pass or get a distinction. The probabilities for the two tries are assumed to be different. 

The above scenario can be modeled by a $\operatorname{MM}(A,S)$ with parameters $(\theta_1,\theta_2,\theta_3,\psi_1,\psi_2,\psi_3)$ such that $|(\theta_1,\theta_2,\theta_3)|=1$ and $|(\psi_1,\psi_2,\psi_3)|=1$, and matrix $A$  
\begin{equation}
\label{eq:sample}
\begin{tabular}{c|cccccc|c}
&$\theta_1$&$\theta_2$&$\theta_3$&$\psi_1$& $\psi_2$&$\psi_3$& $\operatorname{P}(y)$\\
\hline
$y_1$&1&0&0&1&0&0&$\theta_1\psi_1$\\
$y_2$&1&0&0&0&1&0&$\theta_1\psi_2$\\
$y_3$&1&0&0&0&0&1&$\theta_1\psi_3$\\
$y_4$&0&1&0&0&0&0&$\theta_2$\\
$y_5$&0&0&1&0&0&0&$\theta_3$
\end{tabular}
\end{equation}
where for clarity we labelled the columns and the rows with the associated parameters and events, respectively, and reported the atomic probabilities.  The parameters $\theta$'s represent the outcome of the first try, whilst the parameters $\psi$'s of the second one. Parameter vectors $\theta_V$ to be considered are $(\theta_1)$, $(\theta_2)$, $(\psi_1)$, $(\theta_1,\psi_1)$ and $(\theta_2,\psi_1)$. All other subvectors $\theta_V$ of $(\theta_i,\psi_i)_{i\in[3]}$ can be dealt with as one of the cases above by symmetry. For $\operatorname{P},\tilde{\operatorname{P}},\operatorname{Q}\in\operatorname{MM}(A,S)$, we denote with $(\theta_i,\psi_i)_{i\in [3]}$, $(\tilde\theta_i,\tilde\psi_i)_{i\in [3]}$ and $(\bar\theta_i,\bar\psi_i)_{i\in [3]}$ the parameter vectors of $\operatorname{P}$, $\tilde{\operatorname{P}}$ and $\operatorname{Q}$ respectively. In general, the I-divergence from P to Q takes the form
\[
\mathcal{D}(\operatorname{Q}||\operatorname{P})=\sum_{i = 2,3}\bar\theta_i\ln\left(\frac{\bar\theta_i}{\theta_i}\right)+\bar\theta_1\sum_{i\in[3]}\bar\psi_i\ln\left(\frac{\bar\theta_1\bar\psi_i}{\theta_1\psi_i}\right).
\]
The Pythagorean equality in equation (\ref{eq:pitagora}) can then be written as
\begin{equation}
\label{eq:stanco}
\sum_{i=2,3}\bar\theta_i\ln\left(\frac{\bar\theta_i}{\theta_i}\right)+\bar\theta_1\sum_{i\in[3]}\bar\psi_i\ln\left(\frac{\bar\theta_1\bar\psi_i}{\theta_1\psi_i}\right)-\sum_{i=2,3}\bar\theta_i\ln\left(\frac{\bar\theta_i}{\tilde\theta_i}\right)-\bar\theta_1\sum_{i\in[3]}\bar\psi_i\ln\left(\frac{\bar{\theta}_1\bar\psi_i}{\tilde\theta_1\tilde\psi_i}\right)-\sum_{i=2,3}\tilde\theta_i\ln\left(\frac{\tilde\theta_i}{\theta_i}\right)-\tilde\theta_1\sum_{i\in[3]}\tilde\psi_i\ln\left(\frac{\tilde\theta_1\tilde\psi_i}{\theta_1\psi_i}\right)=0.
\end{equation}

Next we look at the form of the above equality for each of the possible varied parameter choices. For each case, we consider only densities $\operatorname{Q}\in L_{\textnormal{sensi}}$ that are usually investigated in sensitivity analysis after a parameter variation.

\begin{enumerate}
\item For $\tilde\theta_V=\tilde\theta_1$, we consider Q such that $\operatorname{Q}\in\textnormal{Slice}(\psi_1,\psi_2,\psi_3)$ and $\bar\theta_1=\tilde\theta_1$. Then $\tilde{\operatorname{P}}$ has parameter vector $(\tilde\theta_i,\psi_i)_{i\in [3]}$, whilst Q has parameters $(\tilde\theta_1,\bar\theta_2,\bar\theta_3,\psi_1,\psi_2,\psi_3)$. Under these conditions and noticing that $\tilde\psi=\bar\psi=\psi$,  (\ref{eq:stanco}) can be simplified to
\begin{equation}
\label{eq:errore}
\sum_{i\in[3]}\bar\theta_i\ln\left(\frac{\bar\theta_i}{\theta_i}\right)-\sum_{i\in[3]}\bar\theta_i\ln\left(\frac{\bar\theta_i}{\tilde\theta_i}\right)-\sum_{i\in [3]}\tilde\theta_i\ln\left(\frac{\tilde\theta_i}{\theta_i}\right)=0.
\end{equation}
By substituting $\tilde\theta_i=\theta_i(1-\tilde\theta_1)/(1-\theta_1)$ into the logarithms, (\ref{eq:errore}) reduces to
\[
\ln\left(\frac{1-\theta_1}{1-\tilde\theta_1}\right)\sum_{i=2,3}(\tilde\theta_i-\bar\theta_i)=0,
\]
which holds for all $\operatorname{Q}\in L_{\textnormal{sensi}}$ since $\sum_{i=2,3}\bar\theta_i=\sum_{i=2,3}\tilde\theta_i=1-\tilde\theta_1$.
\item For $\tilde\theta_V=\tilde\theta_2$, we consider Q such that $\operatorname{Q}\in\textnormal{Slice}(\psi_1,\psi_2,\psi_3)$ and $\bar\theta_2=\tilde\theta_2$. Then $\tilde{\operatorname{P}}$ has parameter vector $(\tilde\theta_i,\psi_i)_{i\in [3]}$, whilst Q has parameters $(\bar\theta_1,\tilde\theta_2,\bar\theta_3,\psi_1,\psi_2,\psi_3)$. Under these conditions and noticing that $\tilde\psi=\bar\psi=\psi$, (\ref{eq:stanco}) can be written as equation (\ref{eq:errore}) which can be simplified as in the previous case to show that the equality holds for all $\operatorname{Q}\in L_{\textnormal{sensi}}$.
\item For $\tilde\theta_V=\tilde\psi_1$, we consider Q such that $\operatorname{Q}\in\textnormal{Slice}(\theta_1,\theta_2,\theta_3)$ and $\bar\psi_1=\tilde\psi_1$. Then $\tilde{\operatorname{P}}$ has parameter vector $(\theta_i,\tilde\psi_i)_{i\in [3]}$, whilst Q has parameters $(\theta_1,\theta_2,\theta_3,\tilde\psi_1,\bar\psi_2,\bar\psi_3)$. Under these conditions and noticing that $\tilde\theta=\theta$,  (\ref{eq:stanco}) can be written as
\[
\theta_1\left(\sum_{i=2,3}\bar\psi_i\ln\left(\frac{\bar\psi_i}{\psi_i}\right)-\sum_{i=2,3}\bar\psi_i\ln\left(\frac{\bar\psi_i}{\tilde\psi_i}\right)-\sum_{i=2,3}\tilde\psi_i\ln\left(\frac{\tilde\psi_i}{\psi_i}\right)\right)=0,
\]
which can be simplified as in the two previous cases to show that the equality holds for all $\operatorname{Q}\in L_{\textnormal{sensi}}$.
\item  For $\tilde\theta_V=(\tilde\theta_1,\tilde\psi_1)$, we consider Q such that $\bar\theta_1=\tilde\theta_1$ and $\bar\psi_1=\tilde\psi_1$, since there are no parameters with index in $F$.  Then $\tilde{\operatorname{P}}$ has parameter vector $(\tilde\theta_i,\tilde\psi_i)_{i\in [3]}$, whilst Q has parameters $(\tilde\theta_1,\bar\theta_2,\bar\theta_3,\tilde\psi_1,\bar\psi_2,\bar\psi_3)$. Under these conditions,  (\ref{eq:stanco}) can be written as
\[
\left(\sum_{i=2,3}\bar\theta_i\ln\left(\frac{\bar\theta_i}{\theta_i}\right)-\bar\theta_i\ln\left(\frac{\bar\theta_i}{\tilde\theta_i}\right)-\tilde\theta_i\ln\left(\frac{\tilde\theta_i}{\theta_i}\right)\right)+\theta_1\left(\sum_{i=2,3}\bar\psi_i\ln\left(\frac{\bar\theta_1\bar\psi_i}{\theta_1\psi_i}\right)-\sum_{i=2,3}\bar\psi_i\ln\left(\frac{\bar\theta_1\bar\psi_i}{\tilde\theta_1\tilde\psi_i}\right)-\sum_{i=2,3}\tilde\psi_i\ln\left(\frac{\tilde\theta_1\tilde\psi_i}{\theta_1\psi_i}\right)\right)=0.
\]
The above equation can be simplified by combining the steps used in the previous cases to show that the equality holds for all $\operatorname{Q}\in L_{\textnormal{sensi}}$.
\item For $\tilde\theta_V=(\tilde\theta_2,\tilde\psi_1)$, we consider Q such that $\bar\theta_2=\tilde\theta_2$ and $\bar\psi_1=\tilde\psi_1$, since there are no parameters with index in $F$. Then $\tilde{\operatorname{P}}$ has parameter vector $(\tilde\theta_i,\tilde\psi_i)_{i\in [3]}$, whilst Q has parameters $(\bar\theta_1,\tilde\theta_2,\bar\theta_3,\tilde\psi_1,\bar\psi_2,\bar\psi_3)$. Under these conditions and using similar steps to the previous points,  (\ref{eq:stanco}) can be written as
\begin{equation}
\label{eq:stanco1}
\ln\left(\frac{1-\tilde\theta_2}{1-\theta_2}\right)(\bar\theta_3-\tilde\theta_3)+ \tilde\psi_1\ln\left(\frac{\tilde\psi_1}{\psi_1}\frac{1-\tilde\theta_2}{1-\theta_2}\right)(\bar\theta_1-\tilde\theta_1) + (1-\tilde\psi_1)\ln\left(\frac{1-\tilde\psi_1}{1-\psi_1}\frac{1-\tilde\theta_2}{1-\theta_2}\right)(\bar\theta_1-\tilde\theta_1) = 0.
\end{equation}
Although the terms in (\ref{eq:stanco}) can be re-arranged differently, the equality cannot be derived and  consequently for the choice $\tilde\theta_V=(\tilde\theta_1,\tilde\psi_1)$ the family of distributions for which the Pythagorean identity holds is restricted.
\end{enumerate}
Notice that the first four choices of parameters corresponded to the analyses introduced in Definition \ref{def:multi}, namely simple in the first three cases and ordered in the fourth, whilst the last choice of parameters does not correspond to any of the newly introduced sensitivity analyses.

Suppose we now refine the model definition with the following probability specifications: $\theta_1=0.2$, $\theta_2=0.5$, $\theta_3=0.3$, $\psi_1=0.4$, $\psi_2=0.4$ and $\psi_3=0.2$. Next consider the choices of parameters varied in points 4 and 5 above. First suppose that $\tilde\theta_1=0.4$ and $\tilde\psi_1=0.2$. For this choice of varied parameters we showed that the Pythagorean identity holds for all $\operatorname{Q}\in L_{\textnormal{sensi}}$. In this case proportional covariation minimizes the I-divergence between the original and the varied distribution, as reported in Figure \ref{fig:ex1}. Consider now case 5  and suppose $\tilde\theta_2=0.3$ and $\tilde\psi_1=0.2$. We showed that the Pythagorean identity does not hold for all $\operatorname{Q}\in L_{\textnormal{sensi}}$. The identity holds in the restricted family of distributions characterized by  (\ref{eq:stanco1}). In this case the I-divergence is not minimized by proportional covariation as reported in Figure \ref{fig:ex2}.

\begin{figure}
    \centering
    \begin{minipage}{.46\textwidth}
        \centering
        \includegraphics[scale=0.35]{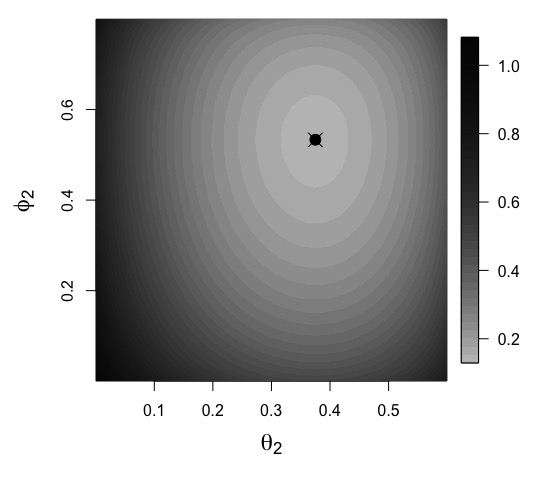}
        \vspace{-0.25cm}
        \caption{I-divergence in case 4 of Section \ref{sec:basta}. The dot represents the minimum  I-divergence and coincides with proportional covariation (star). }
        \label{fig:ex1}
    \end{minipage}%
    \hspace{0.5cm}
    \begin{minipage}{0.46\textwidth}
        \centering
        \includegraphics[scale=0.35]{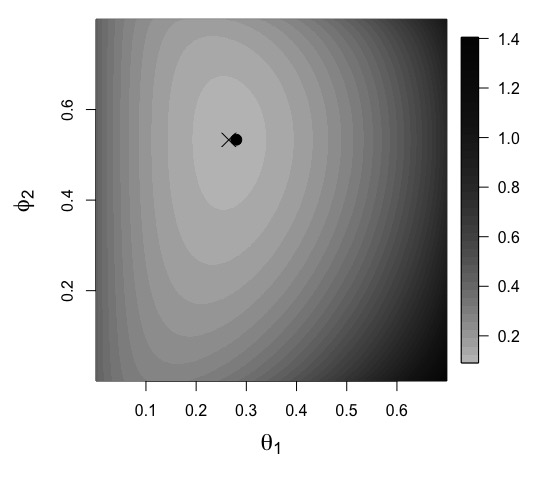}
        \vspace{-0.25cm}
        \caption{I-divergence in case 5 of Section \ref{sec:basta}. The star represents the minimum I-divergence and the dot represents proportional covariation.}
        \label{fig:ex2}
    \end{minipage}
\end{figure}

\subsection{Proportional covariation and I-projections}

Given a $\operatorname{P}\in \textnormal{MM}(A,S)$ and its $\tilde\theta_V$-proportional covariation $\tilde{\operatorname{P}}$, Theorem \ref{theo:pincopallo} identifies the set of distributions  $\textnormal{Q}\in L_{\textnormal{sensi}}$ that satisfy  the Pythagorean identity in (\ref{eq:pitagora}). This is the first step towards the proof of optimality of proportional covariation. For $\emptyset\neq H\subseteq C$ define
\[
\mathbb{Y}_H=\left\{y\in\mathbb{Y}: A_{y,i}=1 \textnormal{ for all } i\in H \textnormal{ and } A_{y,i}=0 \textnormal{ for all } i\in C\setminus H \right\}
\]
and consider those $H$ for which $\mathbb{Y}_H\neq\emptyset.$ These sets $H$ can be determined once the model  $\textnormal{MM}(A,S)$ and the indexes of parameters to be varied, $V$, are known.

For $y\in \mathbb{Y}$, with a slight abuse of notation, if $A_{y,i}=1$ for all $i\in B\subseteq [k]$  we write $\theta_B^{A_{y,B}}=\prod_{i\in B}\theta_i^{A_{y,i}}=\theta_B$. Thus the symbol $\theta_B$ might indicate the vector $(\theta_{i})_{i\in B}$ and the square-free monomial $\prod_{i\in B}\theta_i$. The context clarifies which interpretation applies. In particular for all $y\in\mathbb{Y}_H$ and $B\subseteq H$ we write $\theta_B^{A_{y,B}}=\theta_B$ and $\operatorname{P}(y)=\theta_F^{A_{y,F}}\theta_H$ for any $\operatorname{P}\in\textnormal{MM}(A,S)$.

\begin{theorem}
\label{theo:pincopallo}
Let $\emptyset\neq V\subset [k]$, $\operatorname{P}\in \textnormal{MM}(A,S)$, with parameter vector $(\theta_F,\theta_V,\theta_{C\setminus V})$,  and $\tilde{\operatorname{P}}$ the $\tilde\theta_V$-proportional covariation of P with parameter $\tilde\theta=(\theta_F,\tilde\theta_V,\tilde\theta_{C\setminus V})$. The density $\operatorname{Q}\in L_{\textnormal{sensi}}$ with parameter $\bar{\theta}=(\theta_F,\bar\theta_V,\bar\theta_{C\setminus V})$ satisfies 
\begin{equation}
\label{eq:cond}
\sum_{H\subseteq C,H\neq \emptyset}\tilde\theta_{V\cap H}(\bar\theta_{\{C\setminus V\}\cap H}-\tilde\theta_{\{C\setminus V\}\cap H})\ln\left(\alpha\frac{\tilde\theta_{V\cap H}}{\theta_{V\cap H}}\right)\sum_{y\in\mathbb{Y}_H}\theta_{F}^{A_{y,F}}=0,
\end{equation}
where $\alpha=\prod\limits_{i\in[n]: S_i\cap H \neq \emptyset}\prod\limits_{j\in \{ S_i\setminus V\}\cap H} (1-|\tilde\theta_{V_i}|)/(1-|\theta_{V_i}|)$, if and only if
$
\mathcal{D}(\operatorname{Q}||\operatorname{P})=\mathcal{D}(\operatorname{Q}||\tilde{\operatorname{P}})+\mathcal{D}(\tilde{\operatorname{P}}||\operatorname{Q}).
$
\end{theorem}
The proof is in \ref{sec:prooftheo}. The first condition that the density Q need to respect, i.e. belonging to Slice$(\theta_F)$ and having probabilities $\bar\theta_V=\tilde\theta_V$, is standard and  commonly made in sensitivity analysis. The second condition given in  (\ref{eq:cond}) has one term only depending on the density Q, namely $\bar\theta_{\{C\setminus V\}\cap H}$, whilst all others can be straightforwardly derived once $\tilde\theta_V$ and $\operatorname{P}\in\textnormal{MM}(A,S)$ are given.

\begin{example}
\label{ex:bestia}
Consider the setup of Section \ref{sec:basta}, where it was showed that for the variation of the parameter $\theta_1$ the Pythagorean identity holds. Conversely, for a variation of parameters $\theta_2$ and $\psi_1$ it did not hold. Therefore, the equality in  \ref{eq:cond} must hold in the first case, but not in the second. We next check this is indeed the case.

When $\theta_1$ is varied, the set $C$ of (co)varied parameters is such that $\theta_C=(\theta_1,\theta_2,\theta_3)$. The summation in (\ref{eq:cond}) considers all subsets $H$ of  $C$ such that $\mathbb{Y}_H$ is non-empty: this is the case only for the indexes of $\theta_2$ and $\theta_3$. So (\ref{eq:cond}) can be written as
\begin{equation}
\label{eq:mobasta}
1\cdot (\bar\theta_2-\tilde\theta_2)\ln\left(1\cdot \frac{1-\tilde\theta_1}{1-\theta_1}\right)\cdot  1 + 1\cdot (\bar\theta_2-\tilde\theta_2)\ln\left(1\cdot \frac{1-\tilde\theta_1}{1-\theta_1}\right)\cdot 1.
\end{equation}
Since  $\bar\theta_2+\bar\theta_3=\tilde\theta_2+\tilde\theta_3=1-\tilde\theta_1$, (\ref{eq:mobasta}) is equal to zero, as expected.

When $\theta_2$ and $\psi_1$ are varied, the set $C$ of (co)varied parameters includes all model's parameters and the summation in (\ref{eq:cond}) considers all subsets $H$ of $C$ coinciding with the sample space given in (\ref{eq:sample}). So (\ref{eq:cond}) can be written as the sum of the expressions on the rhs of Table \ref{table:bestia}. It can be noticed that this sum is equal to \ref{eq:stanco1}: the first row of Table \ref{table:bestia} is zero, the second row is equal to the first term in (\ref{eq:stanco1}), the third row is equal to the second term in (\ref{eq:stanco1}) and the sum of the fourth and fifth rows is equal to the third term in (\ref{eq:stanco1}) since $\bar\psi_2+\bar\psi_3=\tilde\psi_2+\tilde\psi_3=1-\tilde\psi_1$. As already noticed, for this choice of parameters the Pythagorean identity does not hold for all distributions in $L_{\textnormal{sensi}}$.

\begin{table}
\centering
\def\arraystretch{2}
\begin{tabular}{|c|c|}
\hline
$\theta_2$&$\displaystyle\tilde\theta_2(1-1)\ln\left(1\cdot \frac{\tilde\theta_2}{\theta_2}\right)\cdot 1$\\
\hline
$\theta_3$& $\displaystyle1\cdot (\bar\theta_3-\tilde\theta_3)\ln\left(\frac{1-\tilde\theta_2}{1-\theta_2}\cdot\frac{1}{1}\right)\cdot 1$\\
\hline
$\theta_1\psi_1$&$\displaystyle\tilde\psi_1(\bar\theta_1-\tilde\theta_1)\ln\left(\frac{\tilde\psi_1}{\psi_1}\frac{1-\tilde\theta_2}{1-\theta_2}\right)\cdot 1$\\
\hline
$\theta_1\psi_2$& $\displaystyle1\cdot(\bar\theta_1\bar\psi_2-\tilde\theta_1\tilde\psi_2)\ln\left(\frac{1-\tilde\psi_1}{1-\psi_1}\frac{1-\tilde\theta_2}{1-\theta_2}\right)\cdot 1$\\
\hline
$\theta_1\psi_3$&$\displaystyle1\cdot(\bar\theta_1\bar\psi_3-\tilde\theta_1\tilde\psi_3)\ln\left(\frac{1-\tilde\psi_1}{1-\psi_1}\frac{1-\tilde\theta_2}{1-\theta_2}\right)\cdot 1$\\
\hline
\end{tabular}
\caption{Elements of the sample space (left column) and contribution to (\ref{eq:cond}) (right column) for the setup in Example \ref{ex:bestia} when $\theta_2$ and $\psi_1$ are varied. \label{table:bestia}}
\end{table}
\end{example}

\begin{corollary}
In the notation of Theorem \ref{theo:pincopallo}, the density $\operatorname{Q}\in L_{\textnormal{sensi}}$ satisfies 
\begin{equation}
\label{eq:cond2}
\sum_{H\subseteq C,H\neq \emptyset}\tilde\theta_{V\cap H}(\bar\theta_{\{C\setminus V\}\cap H}-\tilde\theta_{\{C\setminus V\}\cap H})\ln\left(\alpha\frac{\tilde\theta_{V\cap H}}{\theta_{V\cap H}}\right)\sum_{y\in\mathbb{Y}_H}\theta_{F}^{A_{y,F}} \geq 0,
\end{equation}
 if and only if
$
\mathcal{D}(\operatorname{Q}||\operatorname{P})\geq \mathcal{D}(\operatorname{Q}||\tilde{\operatorname{P}})+\mathcal{D}(\tilde{\operatorname{P}}||\operatorname{Q}).$
\end{corollary}
This result follows by substituting the equalities in the proof of Theorem \ref{theo:pincopallo} with inequalities.

 Since for $\tilde{\operatorname{P}}$, $\operatorname{P}$ and all the distributions $\operatorname{Q}$ characterized by  (\ref{eq:cond2}) the Pythagorean inequality holds, then it can be proven that $\tilde{\operatorname{P}}$ is the I-projection of $\operatorname{P}$ into this well-specified family of distributions. Let
\[
L_{\textnormal{costr}}=L_{\textnormal{sensi}}\cap \left\{Q\in\textnormal{MM}(A,S):\sum_{H\subseteq C,H\neq \emptyset}\tilde\theta_{V\cap H}(\bar\theta_{\{C\setminus V\}\cap H}-\tilde\theta_{\{C\setminus V\}\cap H})\ln\left(\alpha\frac{\tilde\theta_{V\cap H}}{\theta_{V\cap H}}\right)\sum_{y\in\mathbb{Y}_H}\theta_{F}^{A_{y,F}}\geq 0\right\}.
\]
\begin{corollary}
\label{cor:bo}
In the notation of Theorem \ref{theo:pincopallo}, $\tilde{\operatorname{P}}$ is the I-projection of $\operatorname{P}$ in the set $L_{\textnormal{constr}}$.
\end{corollary}
\begin{proof}
Let $\bar{L}$ be the smallest convex and closed subset of $\Delta_{q-1}$ which includes $L_{\textnormal{constr}}$. From Section \ref{sec:proj}, there exists a unique $\operatorname{P}^*\in\bar{L}$ such that   $\mathcal D(\operatorname{Q} || \operatorname{P}) \geq \mathcal D(\operatorname{Q} || \operatorname{P}^\ast )  + \mathcal D(\operatorname{P}^\ast || \operatorname{P})$. But since $\tilde{\operatorname{P}}\in L_{\textnormal{constr}}\subseteq\bar{L}$ satisfies the Pythagorean identity then $\tilde{\operatorname{P}}=\operatorname{P}^*$.
\end{proof}

\citet{Csiszar2004} proved that the I-projection satisfies the Pythagorean identity using the fact that $L$ is closed and convex. Here we took a different approach by taking advantage of the specific monomial form of the statistical models we study. By characterizing the class of distributions for which the Pythagorean identity holds, we have then been able to prove that proportional covariation is the I-projection within this family.

Although Corollary \ref{cor:bo} demonstrates that proportional covariation minimizes the I-divergence between the original distribution and those in the set $L_{\textnormal{constr}}$, it does not provide information on whether $L_{\textnormal{constr}}$ includes all distributions of interest in sensitivity analysis or not. More explicitly, Corollary \ref{cor:bo} does not specify whether, given $\operatorname{P}\in\textnormal{MM}(A,S)$ and $\tilde\theta_V$, $\tilde{\operatorname{P}}$ is the I-projection of $\operatorname{P}$ in $L_{\textnormal{sensi}}$. This is the case if and only if $L_{\textnormal{sensi}}=L_{\textnormal{constr}}$, i.e.  if for all $\textnormal{Q}\in L_{\textnormal{sensi}}$ the condition in (\ref{eq:cond2}) holds.
Theorem \ref{theo:final} below states that for the multi-way analyses in Definition \ref{def:multi}, proportional covariation is indeed the I-projection of the original distribution in the set of all distributions usually considered in sensitivity analysis. Namely for such analyses $L_{\textnormal{sensi}}=L_{\textnormal{constr}}$.

\begin{theorem}
\label{theo:final}
In the notation of Theorem \ref{theo:pincopallo}, if $\tilde\theta_V$ is chosen according to a simple, complete or ordered sensitivity analysis, then $\tilde{\operatorname{P}}$ is the I-projection of $\operatorname{P}$ in $L_{\textnormal{sensi}}$.
\end{theorem}
The proof is given in \ref{proofino}. Notice that the result holds for regular MMs and in particular it holds for all the already mentioned graphical models entertaining a monomial parametrization. Illustrations of this result were given in  Section \ref{sec:basta}: in the first four cases, corresponding to simple or ordered analyses, the Pythagorean identity holds for all $\textnormal{Q}\in L_{\textnormal{sensi}}$ and thus the $\tilde\theta_V$-proportional covariation scheme is the I-projection over the set of all distribution of interest. Conversely, in the fifth case of  Section \ref{sec:basta}, which does not correspond to any of the new multi-way analyses of Definition \ref{def:multi}, the Pythagorean identity holds in a restricted set of distributions, namely $L_{\textnormal{constr}}$. Thus, as specified by Corollary \ref{cor:bo}, proportional covariation is the I-projection over this restricted space only.

 \subsection{BN classifiers}
 \label{sec:class}
BN classifiers are BNs whose graph entertains some specific properties designed for classification problems. BN classifiers have been successfully used in a wide array of real-world applications, with a competitive predictive performance against other classification techniques, because of their intuitiveness and computational efficiency~\citep[see e.g.][for a review]{Bielza2014}. A BN classifier  is defined by partitioning the BN vertex set into the set of features $Fe$ and the classes $Cl$, so that $\mathcal{V}=\{Y_i:i\in  Fe\}\cup\{Y_i:i\in Cl\}$. It is customary that the edge set of a BN classifier is such that feature variables are not allowed to have class children. For simplicity here we focus on univariate classification problems where there is a single class variable. However  our results apply to multidimensional classes since in a BN classifier the class variables can be collapsed into a unique vertex.

\begin{figure}
\begin{center}
\begin{subfigure}{0.32\textwidth}
\centerline{
\scalebox{0.75}{
\xymatrix{
&*+[Fo]{Y_{Cl}}\ar[rd]\ar[d]\ar[ld]&\\
*+[Fo]{Y_{Fe_1}}&*+[Fo]{Y_{Fe_2}}&*+[Fo]{Y_{Fe_3}}\\
&&
}}
}
\vspace{0.22cm}
\caption{\footnotesize{A Naive Bayes classifier.}\label{fig:BNC1}}
\end{subfigure}
\begin{subfigure}{0.32\textwidth}
\centerline{
\scalebox{0.75}{
\xymatrix{
&*+[Fo]{Y_{Cl}}\ar[rd]\ar[d]\ar[ld]&\\
*+[Fo]{Y_{Fe_1}}&*+[Fo]{Y_{Fe_2}}&*+[Fo]{Y_{Fe_3}}\\
&*+[Fo]{Y_{Sp}}\ar[ru]\ar[u]\ar[lu]
}
}}
\caption{\footnotesize{A SPODE classifier with super parent $Y_{Sp}$.}\label{fig:BNC2}}
\end{subfigure}
\begin{subfigure}{0.32\textwidth}
\centerline{
\scalebox{0.75}{
\xymatrix{
&*+[Fo]{Y_{Cl}}\ar[rd]\ar[d]\ar[ld]&\\
*+[Fo]{Y_{Fe_1}}\ar[r]&*+[Fo]{Y_{Fe_2}}\ar[r]&*+[Fo]{Y_{Fe_3}}\\
&&
}}
}
\vspace{0.22cm}
\caption{\footnotesize{A generic BN classifier.}\label{fig:BNC3}}
\end{subfigure}
\end{center}
\vspace{-0.3cm}
\caption{Examples of BN classifiers.}
\end{figure}
 
BN classifiers range from the simplest naive Bayes classifier where the features are conditionally independent given the class variable (given in Figure~\ref{fig:BNC1}), to generic dependence structures between the features (as for example in Figure~\ref{fig:BNC3}). A BN classifier of interest is the super-parent-one-dependence estimator (SPODE)~\citep{Keogh2002} where all features depend on one specific feature called super-parent (see Figure~\ref{fig:BNC2}).
 
 Since BN classifiers are BN models, they can be represented as MMs as shown in Example~\ref{ex:BN}. However, notice that this is not the only monomial representation of a BN classifier~\citep[see e.g.][]{Varando2015,Varando2016}. Since BN classifiers are MMs, we can apply our methodology and deduce the following result.
 
 \begin{theorem}
 \label{theo:CBN}
 Consider  a naive Bayes classifier with features $Y_{Fe_1},\dots,Y_{Fe_m}$. In the notation of Theorem \ref{theo:pincopallo}, if $\tilde\theta_V$ is chosen so that $V\subset\times_{i\in[m]}\mathbb{Y}_{Fe_i}$, then $\tilde{\operatorname{P}}$ is the I-projection of $\operatorname{P}$ in $L_{\textnormal{sensi}}$.
 \end{theorem}

\begin{proof}
 This result follows from Theorem~\ref{theo:final} by noticing that (co)varied parameters conditionally on different values of $Y_{Cl}$ never appear in the same monomial, thus giving a simple sensitivity analysis. For each  instantiation of $Y_{Cl}$, the feature variables are independent, thus giving a complete sensitivity analysis. Since for these two analyses the $\tilde\theta_V$-proportional covariation scheme is optimal the result then follows.
\end{proof}

Thus in a naive Bayes classifier for any choice of conditional probabilities from the feature variables to be varied, proportional covariation is optimal. This result can be extended straightforwardly to SPODE classifiers by excluding the super-parent node from the feature variables set. Then for any variation of probabilities of the other features, proportional covariation is optimal. For generic BN classifiers the optimality of proportional covariation holds for the cases formalized in Theorem~\ref{theo:pincopallo} and Theorem~\ref{theo:final}.

\section{A real-world application}
\label{sec:real}
In this section we study a subset of the dataset of \citet{eisner2011} including metabolomic
information of 77 individuals: 47 of them suffering of cachexia, whilst the remaining do not.
Cachexia is a metabolic syndrome characterized by loss of muscle with or without loss of fat mass. Although the study of \citet{eisner2011} included 71 different metabolics which
could possibly distinguish individuals who suffer of Cachexia from those who do not, for our
illustrative purposes we focus on only six of them: Adipate (A), Betaine (B), Fumarate (F),
Glucose (GC), Glutamine (GM) and Valine (V). These metabolics are measured in a continuous scale and have been recently investigated in the context of Gaussian BNs \citep{Gorgen2020}. The variables are discretized into three levels using the equal frequency method \citep[see e.g.][]{Ropero2018}. A BN over  these variables together with the variable muscle loss (ML , taking values ``yes'' for individuals who suffered of cachexia and ``no'' otherwise) is learnt using  the Bayesian hill-climbing method  implemented in the \texttt{bnlearn} R package \citep{Scutari2010} and reported in Figure \ref{fig:network}. Since there are six ternary and one binary variables, the model has 1458 atomic probabilities. Furthermore, because of the network structure there are 54 parameters, and consequently the BN in Figure \ref{fig:network} could be represented by a MM with associated matrix $A$ of dimension $1458\times 54$.

\begin{figure}
\centering
\includegraphics[scale=0.5]{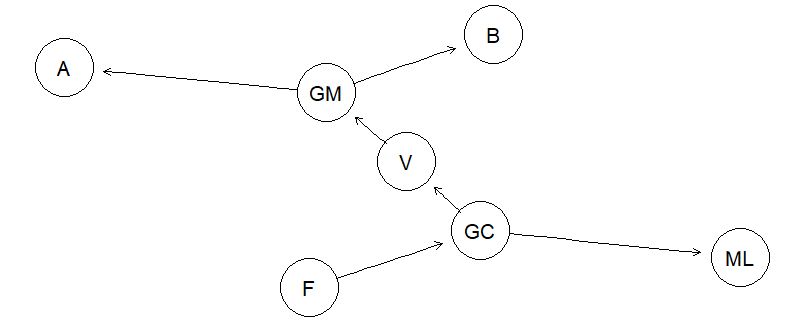}
\caption{Learnt BN for the cachexia dataset using hill-climbing. \label{fig:network}}
\end{figure}

\citet{eisner2011} reports that the variable Adipate (A) has the largest mutual information with muscle loss (ML). The estimated CPT of ML given A and computed using the \texttt{gRain} package \citep{soren} is 
\begin{center}
\def\arraystretch{1.6}
\begin{tabular}{c|c|c|c}

MC / A & low & average & high \\ 
\hline
yes & 0.49 & 0.64 & 0.70\\
\hline
no & 0.51 & 0.36 & 0.30\\

\end{tabular}
\end{center}
As the level of Adipate increases, the probability of having the disease increases and the probability of having the disease given a high level of Adipate is estimated as 0.70. Given that Adipate is marginally a powerful predictor for muscle loss, due to the largest value of mutual information, it is of particular interest to estimate correctly the probability of ML = yes given A = high and to investigate the effect of misspecifications of this probability. Using the \texttt{sensquery} function of \texttt{bnmonitor} the required individual parameter changes to obtain a conditional probability of either 0.68 or 0.72 are computed together with the associated CD distances and I-divergences. These are reported in Table \ref{tavola}. The software reports only one parameter change per CPT and uses proportional covariation, since the scheme is optimal for one-way analyses. There are multiple parameter changes that independently meet the constraint of the probability of ML = yes given A = high being equal to either 0.72 or 0.68. For the two cases the CD distances are comparable, but the I-divergence is overall slightly larger when the conditional probability is decreased. It can be further noticed that in order to achieve a change of $\pm 0.02$ in the conditional probability all parameters need to be changed by more than 0.1, with the exception of F = low for an increase to 0.72.

\begin{table}
\begin{center}
\begin{tabular}{|c|c|c|c|c|c|}
\toprule
\multicolumn{6}{c}{P(ML = yes $|$ A = high) = 0.72}\\
\midrule 
Parameter & Old Value & New Value  & Abs. Diff. & CD Dist. & I-Diver. \\
\midrule
P(F = low) & 0.34 & 0.25 & 0.09& 0.42 & 0.0018 \\
P(ML = yes $|$ GC = low) & 0.39 & 0.51 & 0.12 & 0.50 & 0.011\\
P(GC = high $|$ F = low) & 0.04 & 0.14 & 0.10 &1.31  & 0.025\\
P(V = low $|$ GC = average) & 0.08 & 0.35 & 0.27& 1.79 & 0.091\\
P(GM = low $|$ V = low) & 0.80 & 0.97 & 0.17 & 2.08 & 0.043\\
P(A = low $|$ GM = average) & 0.19 & 0.67 & 0.48 &2.13 & 0.181 \\
\bottomrule
\multicolumn{6}{c}{}\\
\toprule
\multicolumn{6}{c}{P(ML = ys $|$ A = high) = 0.68}\\
\midrule 
Parameter & Old Value & New Value & Abs. Diff. & CD Dist. & I-Diverg. \\
\midrule
P(F = low) & 0.34 & 0.44 & 0.10 & 0.43 & 0.022 \\
P(ML = yes $|$ GC = low) & 0.39 & 0.23 & 0.16 &0.74 & 0.019\\
P(GM = low $|$ V = low) & 0.80 & 0.55 & 0.25 &1.19 & 0.055\\
P(V = low $|$ GC = low) & 0.88 & 0.58 & 0.30 &1.65 & 0.094\\
P(A = low $|$ GM = low) & 0.79 & 0.34 & 0.45 &2.01 & 0.156 \\
P(GC = average $|$ F = low) & 0.19 & 0.72 &0.53 & 2.37 & 0.219\\
\bottomrule
\end{tabular}
\end{center}
\caption{Required parameter changes in one-way sensitivity analysis  to entertain constraints on the conditional probability of ML = yes given A = high. Columns: Parameter - varied probability; Old Value - original probability value; New Value -  required probability change; Abs. Diff. - absolute value of the difference between Old Value and New Value; CD Dist. - CD distance between the original and the varied BN; I-Diver. - I-divergence between the original and the varied BN. \label{tavola}}
\end{table}

For ease of exposition we next consider only 2-way sensitivity analyses considering the parameters that \texttt{bnmonitor} suggested changing in one-way sensitivity analysis. Table \ref{quasi} reports the results for every pair of parameters varied and using proportional covariation. The following conclusions can be made:
\begin{itemize}
\item the required changes to the parameters are smaller in absolute value compared to one-way analyses;
\item the CD distance and the I-divergence takes a value inbetween those of the one-way analysis of the corresponding parameters;
\item the vast majority of pairs of varied parameters correspond to the novel sensitivity analyses introduced in Section \ref{sec:multicov}. For all these choices of parameters, Theorem \ref{theo:final} guarantees that proportional covariation minimizes the I-divergence between the original and the varied BNs;
\item five out of the 30 pairs of varied parameters do not correspond to any of the novel sensitivity methods of Section \ref{sec:multicov}. For such pairs, it can be shown that indeed proportional covariation does minimize neither the CD distance nor the I-divergence.
\end{itemize}

\begin{table}
\begin{center}
\begin{tabular}{|c|c|c|c|c|c|}
\toprule
\multicolumn{6}{c}{P(ML = cachexic $|$ A = high) = 0.72}\\
\midrule 
Parameter1 & Parameter2 & Type & Avg. Change & CD Dist. & I-Diver.\\
\midrule
P(F = low) & P(ML = yes $|$ GC = low) & complete &  0.06&0.50 & 0.008\\ 
P(F = low) & P(GC = high $|$ F = low) & ordered & 0.05& 1.03 & 0.015\\ 
P(F = low) & P(V = low $|$ GC = average) & complete & 0.06 & 0.78 & 0.017\\
 P(F = low)& P(GM = low $|$ V = low)  & complete & 0.05 & 0.48 & 0.018\\
 P(F = low)& P(A = low $|$ GM = average)  & complete & 0.19 & 1.71 & 0.102 \\
 P(ML = yes $|$ GC = low) & P(GC = high $|$ F = low) & \textbf{none} & 0.06 & 0.47 & 0.009\\
 P(ML = yes $|$ GC = low) &P(V = low $|$ GC = average) & simple & 0.09 & 1.07 & 0.028\\
  P(ML = yes $|$ GC = low) & P(GM = low $|$ V = low) & complete & 0.08 & 1.50 & 0.027\\
 P(ML = yes $|$ GC = low) &P(A = low $|$ GM = average)  & complete & 0.10 & 0.95 & 0.016\\ 
P(GC = high $|$ F = low) & P(V = low $|$ GC = average) & \textbf{none} & 0.12 & 1.57 & 0.064\\
P(GC = high $|$ F = low) & P(GM = low $|$ V = low) &  complete & 0.08 & 1.73 & 0.024\\
P(GC = high $|$ F = low) & P(A = low $|$ GM = average) & complete & 0.12 & 1.92 & 0.042 \\
P(V = low $|$ GC = average) & P(GM = low $|$ V = low) & ordered & 0.13 & 1.70 & 0.079\\
P(V = low $|$ GC = average) &P(A = low $|$ GM = average)  & complete & 0.19 & 2.33 & 0.079 \\
 P(GM = low $|$ V = low) & P(A = low $|$ GM = average)& \textbf{none} & 0.15 & 1.93 & 0.047\\
\bottomrule
\multicolumn{6}{c}{}\\
\toprule
\multicolumn{6}{c}{P(ML = cachexic $|$ A = high) = 0.68}\\
\midrule 
Parameter1 & Parameter2 & Type & Avg. Change & CD Dist. & I-Diver.\\
\midrule
P(F = low) & P(ML = yes $|$ GC = low) & complete & 0.07 & 0.69 & 0.015\\
P(F = low) & P(GM = low $|$ V = low) & complete & 0.10 & 1.05 & 0.035\\
P(F = low) & P(V = low $|$ GC = low) & complete & 0.06 & 0.63 & 0.020\\
P(F = low) &P(A = low $|$ GM = low) & complete & 0.25 & 0.64 & 0.020\\
P(F = low) & P(GC = average $|$ F = low) & ordered & 0.23 & 2.00 & 0.166 \\
P(ML = yes $|$ GC = low) & P(GM = low $|$ V = low) & complete & 0.08 & 0.83 & 0.012\\
P(ML = yes $|$ GC = low) & P(V = low $|$ GC = low) & complete& 0.09 & 1.08 & 0.017\\
P(ML = yes $|$ GC = low) & P(A = low $|$ GM = low) & complete & 0.19 & 1.70 & 0.100\\
P(ML = yes $|$ GC = low) & P(GC = average $|$ F = low) & \textbf{none} & 0.23 & 1.93 & 0.124\\
P(GM = low $|$ V = low)  & P(V = low $|$ GC = low) & ordered  & 0.15 & 1.15 & 0.044\\
P(GM = low $|$ V = low) &P(A = low $|$ GM = low) & ordered & 0.16 & 1.15 & 0.047\\
P(GM = low $|$ V = low)  & P(GC = average $|$ F = low) & complete & 0.21 & 2.05 & 0.074\\
P(V = low $|$ GC = low) & P(A = low $|$ GM = low) & complete & 0.19 & 2.12 & 0.070\\
P(V = low $|$ GC = low) & P(GC = average $|$ F = low) & \textbf{none} & 0.21 & 1.64 & 0.081\\
P(A = low $|$ GM = low) &P(GC = average $|$ F = low) & complete & 0.27 & 2.55 & 0.157 \\
\bottomrule
\end{tabular}
\end{center}
\caption{All possible 2-way sensitivity analyses for the varied parameters in Table \ref{tavola}. Columns: Parameter1 - first varied probability; Parameter2 - second varied probability; Type - type of sensitivity analysis (either simple, complete, ordered or none); Avg. Change - mean of the absolute difference between the old and the new probabilities; CD Dist. - CD distance between the original and the varied BN; I-Diver. - I-divergence between the original and the varied BN. \label{quasi}}
\end{table}
\section{Discussion}
\label{sec:discussion}

The representation of a wide array of statistical models in terms of the defining atomic monomial probabilities has proven useful for a variety of applications, including sensitivity analysis. In this paper, we took advantage of this representation to develop a  formal geometric approach for sensitivity analysis which uses elements of information geometry. This approach has enabled us to demonstrate the optimality of proportional covariation for novel multi-way choices of varied parameters defined by the characteristics of the monomial atomic probabilities. Attention was 
devoted to BN classifiers where the tuning of the feature probabilities is often critical to ensure the classifier works reliably.

Although in this work we focused on models having multilinear atomic probabilities, our geometric approach could be used to investigate more general classes of models, for instance dynamic BNs whose atomic probabilities are not necessarily multilinear. Preliminary results   suggest that the I-divergence exhibit different properties than in the multilinear case, with the potential of even more informative sensitivity investigations.

{
We concentrated  on I-divergences but other measures of closeness between distributions could have been considered, for instance the already mentioned $\phi$-divergences and CD distances. It is yet unknown whether our newly introduced covariation schemes would be optimal under these other measures. 

The \texttt{bnmonitor} R package was used in a real-world application to understand the relationship between various metabolics and muscle loss. However, its current implementation only provides a user-friendly implementation of one-way sensitivity methods. In a future version of \texttt{bnmonitor} we plan to also implement multi-way methods to provide a more comprehensive toolbox for sensitivity analysis in BNs.

\bibliographystyle{plainnat} 
\bibliography{bib1}


\appendix

\section{Proofs}
\label{sec:proof}

\subsection{Proof of Theorem \ref{theo:pincopallo}}
\label{sec:prooftheo}
Substituting $\bar{\theta}_V=\tilde\theta_V$ and $\textnormal{Q}\in\textnormal{Slice}(\theta_F)$, we can write $\mathcal{D}(\textnormal{Q}||\textnormal{P})$ as 
\begin{equation}
\label{eq:1}
\mathcal{D}(\textnormal{Q}||\textnormal{P})=\sum_{y\in\mathbb{Y}}\theta_F^{A_{y,F}}\tilde\theta_V^{A_{y,V}}\bar\theta_F^{A_{y,C\setminus V}}\ln\frac{\theta_F^{A_{y,F}}\tilde\theta_V^{A_{y,V}}\bar\theta_F^{A_{y,C\setminus V}}}{\theta_F^{A_{y,F}}\theta_V^{A_{y,V}}\theta_F^{A_{y,C\setminus V}}}=\sum_{y\in\mathbb{Y}}\theta_F^{A_{y,F}}\tilde\theta_V^{A_{y,V}}\bar\theta_F^{A_{y,C\setminus V}}\ln\frac{\tilde\theta_V^{A_{y,V}}\bar\theta_F^{A_{y,C\setminus V}}}{\theta_V^{A_{y,V}}\theta_F^{A_{y,C\setminus V}}}.
\end{equation}
For all $\emptyset\neq H\subset[k]$,  define $
\mathbb{Y}^{=}_H=\{y\in\mathbb{Y}:A_{y,i}=0,\mbox{ for all }i\in H\}.
$
Now (\ref{eq:1}) can be split as 
\begin{equation}
\label{eq:2}
\mathcal{D}(\textnormal{Q}||\textnormal{P})=\sum_{y\in\mathbb{Y}\setminus \mathbb{Y}_C^{=}}\theta_F^{A_{y,F}}\tilde\theta_V^{A_{y,V}}\bar\theta_{C\setminus V}^{A_{y,C\setminus V}}\ln\frac{\tilde\theta_V^{A_{y,V}}\bar\theta_{C\setminus V}^{A_{y,C\setminus V}}}{\theta_V^{A_{y,V}}\theta_{C\setminus V}^{A_{y,C\setminus V}}}+\sum_{y\in\mathbb{Y}_C^{=}}\theta_F^{A_{y,F}}\tilde\theta_V^{A_{y,V}}\bar\theta_{C\setminus V}^{A_{y,C\setminus V}}\ln\frac{\tilde\theta_V^{A_{y,V}}\bar\theta_{C\setminus V}^{A_{y,C\setminus V}}}{\theta_V^{A_{y,V}}\theta_{C\setminus V}^{A_{y,C\setminus V}}},
\end{equation}
but since for all $y\in\mathbb{Y}_C^{=}$ and $i\in C$ $A_{y,i}=0$, the second term on the rhs of  (\ref{eq:2}) is equal to zero. The set $\mathbb{Y}\setminus\mathbb{Y}_C^{=}$ includes all events $y$ for which $A_{y,i}=1$ for at least one $i\in C$. Thus $\mathbb{Y}\setminus\mathbb{Y}_C^{=}=\bigcup\limits_{H\subseteq C, H\neq\emptyset}\mathbb{Y}_H$, recalling that $\mathbb{Y}_H$ is the set of events $y$ for which $A_{y,i}=1$ for $i\in H$ and $A_{y,i}=0$ for $i\in C\setminus H$. Furthermore since these sets $\mathbb{Y}_H$, for $H\subseteq C$, are disjoint we have that 
\begin{equation}
\label{eq:3}
\mathcal{D}(\textnormal{Q}||\textnormal{P})=\sum_{H\subseteq C, H\neq \emptyset}\sum_{y\in\mathbb{Y}_H}\theta_F^{A_{y,F}}\tilde\theta_V^{A_{y,V}}\bar\theta_{C\setminus V}^{A_{y,C\setminus V}}\ln\frac{\tilde\theta_V^{A_{y,V}}\bar\theta_{C\setminus V}^{A_{y,C\setminus V}}}{\theta_V^{A_{y,V}}\theta_{C\setminus V}^{A_{y,C\setminus V}}},
\end{equation}
where terms in the internal sum are only for $\mathbb{Y}_H\neq 0$. For any $H\subseteq C$, $\textnormal{P}\in \textnormal{MM}(A,S)$ and $y\in\mathbb{Y}_H$, by multilinearity it holds
\begin{equation}
\label{eq:3a}
\theta_V^{A_{y,V}}=\prod_{i\in V\cap H}\theta_i= \theta_{V\cap H},\hspace{1cm}
\theta_{C\setminus V}^{A_{y,C\setminus V}}=\prod_{i\in\{C\setminus V\}\cap H}\theta_i=\theta_{\{C\setminus V\}\cap H}.
\end{equation}
Substituting (\ref{eq:3a}) and using properties of the logarithm,  (\ref{eq:3}) simplifies to
\begin{equation}
\label{eq:4}
\mathcal{D}(\textnormal{Q}||\textnormal{P})=\sum_{\substack{H\subseteq C, H\neq \emptyset\\y\in\mathbb{Y}_H}}\theta_F^{A_{y,F}}\tilde\theta_{V\cap H}\bar\theta_{\{C\setminus V\}\cap H}\ln\frac{\tilde\theta_{V\cap H}}{\theta_{V\cap H}}+\sum_{\substack{H\subseteq C, H\neq \emptyset\\y\in\mathbb{Y}_H}}\theta_F^{A_{y,F}}\tilde\theta_{V\cap H}\bar\theta_{\{C\setminus V\}\cap H}\ln\frac{\bar\theta_{\{C\setminus V\}\cap H}}{\theta_{\{C\setminus V\}\cap H}}.
\end{equation}
Analogously 
\begin{equation}
\label{eq:5}
\mathcal{D}(\textnormal{Q}||\tilde{\textnormal{P}})=\sum_{H\subseteq C, H\neq \emptyset}\sum_{y\in\mathbb{Y}_H}\theta_F^{A_{y,F}}\tilde\theta_{V\cap H}\bar\theta_{\{C\setminus V\}\cap H}\ln\frac{\bar\theta_{\{C\setminus V\}\cap H}}{\tilde\theta_{\{C\setminus V\}\cap H}},
\end{equation}
\begin{equation}
\label{eq:6}
\mathcal{D}(\tilde{\textnormal{P}}||\textnormal{P})=\sum_{\substack{H\subseteq C, H\neq \emptyset\\y\in\mathbb{Y}_H}}\theta_F^{A_{y,F}}\tilde\theta_{V\cap H}\tilde\theta_{\{C\setminus V\}\cap H}\ln\frac{\tilde\theta_{V\cap H}}{\theta_{V\cap H}}+\sum_{\substack{H\subseteq C, H\neq \emptyset\\y\in\mathbb{Y}_H}}\theta_F^{A_{y,F}}\tilde\theta_{V\cap H}\tilde\theta_{\{C\setminus V\}\cap H}\ln\frac{\tilde\theta_{\{C\setminus V\}\cap H}}{\theta_{\{C\setminus V\}\cap H}}.
\end{equation}
In  (\ref{eq:5}) we used the assumption that $\bar{\theta}_V=\tilde\theta_V$ and in (\ref{eq:5}) and (\ref{eq:6}) we used Theorem \ref{theo:pinco}.
Next we use  the fact that $\tilde\theta_{\{C\setminus V\}\cap H}$ is computed via proportional covariation. For $H\subseteq C$ it holds that 
\begin{align}
\tilde\theta_{\{C\setminus V\}\cap H}&=\prod_{\substack{i\in[n]:\\ S_i\cap H \neq \emptyset}}\prod_{j\in \{ S_i\setminus V\}\cap H} \frac{1-|\tilde\theta_{V_i}|}{1-|\theta_{V_i}|}\theta_{\{S_i\setminus V\}\cap H}=\left(\prod_{\substack{i\in[n]:\\ S_i\cap H \neq \emptyset}}\prod_{j\in \{ S_i\setminus V\}\cap H} \frac{1-|\tilde\theta_{V_i}|}{1-|\theta_{V_i}|}\right)\theta_{\{C\setminus V\}\cap H} \nonumber\\
&= \alpha \theta_{\{C\setminus V\}\cap H} \label{eq:7}
\end{align} 
Substituting (\ref{eq:7}) into the logarithms in (\ref{eq:5}) and (\ref{eq:6}) and re-arranging the factors yields
\begin{equation}
\label{eq:8}
\mathcal{D}(\textnormal{Q}||\tilde{\textnormal{P}})=\sum_{\substack{H\subseteq C, H\neq \emptyset\\y\in\mathbb{Y}_H}}\theta_F^{A_{y,F}}\tilde\theta_{V\cap H}\bar\theta_{\{C\setminus V\}\cap H}\ln\frac{\bar\theta_{\{C\setminus V\}\cap H}}{\theta_{\{C\setminus V\}\cap H}}-\sum_{\substack{H\subseteq C, H\neq \emptyset\\y\in\mathbb{Y}_H}}\theta_F^{A_{y,F}}\tilde\theta_{V\cap H}\bar\theta_{\{C\setminus V\}\cap H}\ln\alpha
\end{equation}
\begin{equation}
\label{eq:9}
\mathcal{D}(\tilde{\textnormal{P}}||\textnormal{P})=\sum_{\substack{H\subseteq C, H\neq \emptyset\\y\in\mathbb{Y}_H}}\theta_F^{A_{y,F}}\tilde\theta_{V\cap H}\tilde\theta_{\{C\setminus V\}\cap H}\ln\frac{\tilde\theta_{V\cap H}}{\theta_{V\cap H}}+\sum_{\substack{H\subseteq C, H\neq \emptyset\\y\in\mathbb{Y}_H}}\theta_F^{A_{y,F}}\tilde\theta_{V\cap H}\tilde\theta_{\{C\setminus V\}\cap H}\ln\alpha
\end{equation}
At this stage the result is proven if under the condition in  (\ref{eq:cond}) the rhs of  (\ref{eq:4}) is equal to the sum of the rhs of (\ref{eq:8}) and (\ref{eq:9}). Since the second term on the rhs of  (\ref{eq:4}) is equal to the first term on the rhs of  (\ref{eq:8}), we can write $\mathcal{D}(\operatorname{Q}||\operatorname{P})=\mathcal{D}(\operatorname{Q}||\tilde{\operatorname{P}})+\mathcal{D}(\tilde{\operatorname{P}}||\operatorname{P})$ as 
\begin{multline}
\label{eq:10}
\sum_{\substack{H\subseteq C, H\neq \emptyset\\y\in\mathbb{Y}_H}}\theta_F^{A_{y,F}}\tilde\theta_{V\cap H}\bar\theta_{\{C\setminus V\}\cap H}\ln\frac{\tilde\theta_{V\cap H}}{\theta_{V\cap H}}+\sum_{\substack{H\subseteq C, H\neq \emptyset\\y\in\mathbb{Y}_H}}\theta_F^{A_{y,F}}\tilde\theta_{V\cap H}\bar\theta_{\{C\setminus V\}\cap H}\ln\alpha\\
-\sum_{\substack{H\subseteq C, H\neq \emptyset\\y\in\mathbb{Y}_H}}\theta_F^{A_{y,F}}\tilde\theta_{V\cap H}\tilde\theta_{\{C\setminus V\}\cap H}\ln\frac{\tilde\theta_{V\cap H}}{\theta_{V\cap H}}-\sum_{\substack{H\subseteq C, H\neq \emptyset\\y\in\mathbb{Y}_H}}\theta_F^{A_{y,F}}\tilde\theta_{V\cap H}\tilde\theta_{\{C\setminus V\}\cap H}\ln\alpha=0.
\end{multline}
By rearranging the terms in (\ref{eq:10}) we have that
\begin{equation}
 \sum_{\substack{H\subseteq C, H\neq \emptyset\\y\in\mathbb{Y}_H}}\theta_F^{A_{y,F}}\tilde\theta_{V\cap H}\left(\bar\theta_{\{C\setminus V\}\cap H}\ln\frac{\tilde\theta_{V\cap H}}{\theta_{V\cap H}}+\bar\theta_{\{C\setminus V\}\cap H}\ln\alpha+\tilde\theta_{\{C\setminus V\}\cap H}\ln\frac{\tilde\theta_{V\cap H}}{\theta_{V\cap H}}+\tilde\theta_{\{C\setminus V\}\cap H}\ln\alpha\right)=0,
\end{equation}
which yields
\begin{equation}
\label{eq:last}
 \sum_{\substack{H\subseteq C, H\neq \emptyset\\y\in\mathbb{Y}_H}}\theta_F^{A_{y,F}}\tilde\theta_{V\cap H}(\bar\theta_{\{C\setminus V\}\cap H}-\tilde\theta_{\{C\setminus V\}\cap H})\ln\left(\alpha\frac{\tilde\theta_{V\cap H}}{\theta_{V\cap H}}\right)=0.
\end{equation}
Noticing that (\ref{eq:last}) equals (\ref{eq:cond}), since only $\theta_F^{A_{y,F}}$ depends on the event $y\in\mathbb{Y}_H$,  proves the result.

\subsection{Proof of Theorem \ref{theo:final}}
\label{proofino}
The result is proven if the condition in (\ref{eq:cond}) holds for all $\textnormal{Q}\in L$. For a simple analysis, for all $i,j\in C$, the monomial $\theta_i\theta_j$ does not divide $\theta^{A_y}$ for any $y\in\mathbb{Y}$. Thus all sets $H$ in condition (\ref{eq:cond}) that need to be considered, i.e. those such that $\mathbb{Y}_H$ is non-empty, have one element only because of regularity. If $H$ is an element of $V$ then $\bar{\theta}_{\{C\setminus V\}\cap H}-\tilde\theta_{\{C\setminus V\}\cap H}=0$ by construction and the result thus follows. Conversely, if $H$ is an element of $C\setminus V$, condition (\ref{eq:cond}) holds if and only if
\begin{equation}
\label{eq:ciaone}
\sum_{j\in C\setminus V} \bar\theta_j-\tilde\theta_j=0.
\end{equation}
Next, (\ref{eq:ciaone}) can be rewritten as
\[
\sum_{i\in[n]:V_i\neq \emptyset} 1-|\tilde\theta_{V_i}|-1+|\tilde\theta_{V_i}|=0,
\]
which is always true. This proves Theorem \ref{theo:final} for simple analyses.

In a complete sensitivity analysis all sets $H$ in condition (\ref{eq:cond}) that need to be considered, i.e. those such that $\mathbb{Y}_H$ is non-empty, are in $\times_{i\in [n]:V_i\neq\emptyset}S_i$, by regularity. Thus, (\ref{eq:cond}) can be written as
\begin{equation}
 \sum_{\substack{H\in \times_{i\in[n]:V_i\neq\emptyset}S_i,\\ H\neq \emptyset}}\tilde\theta_{V\cap H}(\bar\theta_{\{C\setminus V\}\cap H}-\tilde\theta_{\{C\setminus V\}\cap H})\ln\left(\alpha\frac{\tilde\theta_{V\cap H}}{\theta_{V\cap H}}\right)\sum_{y\in\mathbb{Y}_H}\theta_{F}^{A_{y,F}}=0.
 \label{eq:prova0}
\end{equation}
Suppose with no loss of generality that the sets $S_i$ such that $V_i\neq \emptyset$ are those with index in the set $[r]$, $r\leq n$. Notice that
\begin{equation}
\bigtimes_{i\in [r]}S_i=\bigcup_{R\subseteq [r]}\left\{\bigtimes_{i\in R}V_i\bigtimes\bigtimes_{i\in [r]\setminus R} \{C_i\setminus V_i\}\right\}.
\label{eqeqeq}
\end{equation}
Thus the result is proven if  (\ref{eq:prova0}) holds for each $R\subseteq [r]$, i.e. if 
\begin{equation}
 \sum_{\substack{H\in R}}\sum_{J\in [r]\setminus R}\tilde\theta_{V\cap H}(\bar\theta_{\{C\setminus V\}\cap J}-\tilde\theta_{\{C\setminus V\}\cap J})\ln\left(\alpha\frac{\tilde\theta_{V\cap H}}{\theta_{V\cap H}}\right)\sum_{y\in\mathbb{Y}_{\{H\cup J\}}}\theta_{F}^{A_{y,F}}=0,
\label{eq:prova1}
\end{equation}
for $H$ and $J$ such that $\mathbb{Y}_{\{H\cup J\}}\neq\emptyset$. First notice that if $R=[r]$, then $\bar\theta_{\{C\setminus V\}\cap J}-\tilde\theta_{\{C\setminus V\}\cap J}=0$ by construction and the result follows. Now fix an $R\subset [r]$ and suppose $k\in [r]\setminus R$. So, (\ref{eq:prova1}) can be written as
\begin{equation}
\label{eq:prova2}
 \sum_{\substack{H\in R}}\sum_{J\in [r]\setminus R\setminus \{k\}}\sum_{j\in C_k\setminus V_k}\tilde\theta_{V\cap H}(\bar\theta_{\{C\setminus V\}\cap J}\bar{\theta}_j-\tilde\theta_{\{C\setminus V\}\cap J}\tilde\theta_{j})\ln\left(\alpha\frac{\tilde\theta_{V\cap H}}{\theta_{V\cap H}}\right)\sum_{y\in\mathbb{Y}_{\{H\cup J\cup \{j\}\}}}\theta_{F}^{A_{y,F}}=0.
\end{equation}
Noticing that $\sum_{j\in C_k\setminus V_k}\bar\theta_j=\sum_{j\in C_k\setminus V_k}\tilde\theta_j=1-|\tilde\theta_{V_k}|$,  (\ref{eq:prova2}) can be rearranged as
\begin{equation}
\label{eq:prova3}
\sum_{\substack{H\in R}}\sum_{J\in [r]\setminus R\setminus \{k\}}\tilde\theta_{V\cap H}\left(\bar\theta_{\{C\setminus V\}\cap J}(1-|\tilde\theta_{V_k}|)-\tilde\theta_{\{C\setminus V\}\cap J}(1-|\tilde\theta_{V_k}|)\right)\ln\left(\alpha\frac{\tilde\theta_{V\cap H}}{\theta_{V\cap H}}\right)\sum_{y\in\mathbb{Y}_{\{H\cup J\cup \{j\}\}}}\theta_{F}^{A_{y,F}}=0.
\end{equation}
By applying the same steps as in  (\ref{eq:prova2})-(\ref{eq:prova3}) for all $k\in [r]\setminus R$, we have that 
\begin{equation}
\sum_{\substack{H\in R}}\tilde\theta_{V\cap H}\left(\prod_{k\in [r]\setminus R}(1-|\tilde\theta_{V_k}|)-\prod_{k\in [r]\setminus R}(1-|\tilde\theta_{V_k}|)\right)\ln\left(\alpha\frac{\tilde\theta_{V\cap H}}{\theta_{V\cap H}}\right)\sum_{y\in\mathbb{Y}_{\{H\cup J\cup \{j\}\}}}\theta_{F}^{A_{y,F}}=0,
\end{equation}
which always holds, thus proving Theorem \ref{theo:final} for complete analyses.

The proof for ordered analyses follows from the one of complete sensitivity analyses by noticing that the sets $R\subseteq [r]$ for which condition (\ref{eq:prova1}) needs to hold is a subset of those already demonstrated in the complete case.
\end{document}